\documentclass[reqno]{amsart}
\usepackage[a4paper]{geometry}
\usepackage[T1]{fontenc}
\usepackage[utf8]{inputenc}
\usepackage{amssymb}
\usepackage[mathscr]{eucal}
\usepackage{mathtools}
\usepackage{enumerate}
\usepackage{accents}
\usepackage{dsfont}
\usepackage{color}
\usepackage[colorlinks=true]{hyperref}
\hypersetup{urlcolor=blue,citecolor=red,linkcolor=blue}
\usepackage[initials]{amsrefs}
\numberwithin{equation}{section}
\theoremstyle{plain}
\newtheorem{theorem}{Theorem}[section]
\newtheorem{proposition}[theorem]{Proposition}
\newtheorem{corollary}[theorem]{Corollary}

\theoremstyle{definition}
\newtheorem*{rh-pb*}{Basic RH problem}
\newtheorem*{rh-I*}{RH problem normalized as $I$ at $\infty$}
\newtheorem*{sol-rh-pb*}{Soliton RH problem}
\newtheorem*{data*}{Data of this RH problem associated with $\BS{u_0(x)}$}
\theoremstyle{remark}
\newtheorem{remark}[theorem]{Remark}
\newtheorem*{notations*}{Notations}
\providecommand{\BS}[1]{\boldsymbol{#1}}  
\providecommand{\D}[1]{\mathbb{#1}}
\newcommand{\dd}{\mathrm{d}}
\newcommand{\eul}{\mathrm{e}}
\newcommand{\ii}{\mathrm{i}}
\newlength{\dhatheight}

\providecommand{\accol}[1]{\lbrace#1\rbrace}

\renewcommand{\Im}{\operatorname{Im}}
\newcommand{\ord}{\mathrm{O}}
\DeclareMathOperator{\Res}{Res}
\newif\ifshort
\shorttrue
\begin{document}
\title[Riemann--Hilbert approach for the mCH equation]{A Riemann--Hilbert approach to the modified Camassa--Holm equation with step-like boundary conditions}
\author[I.~Karpenko]{Iryna Karpenko}
\address{B.~Verkin Institute for Low Temperature Physics and Engineering\\
47 Nauky Avenue\\ 61103 Kharkiv\\ Ukraine\\ and
Faculty of Mathematics\\ University of Vienna\\
Oskar-Morgenstern-Platz 1\\ 1090 Wien\\ Austria}
\email{inic.karpenko@gmail.com}
\author[D.~Shepelsky]{Dmitry Shepelsky}
\address{B.~Verkin Institute for Low Temperature Physics and Engineering\\
47 Nauky Avenue\\ 61103 Kharkiv\\ Ukraine}
\email{shepelsky@yahoo.com}
\author[G.~Teschl]{Gerald Teschl}
\address{Faculty of Mathematics\\ University of Vienna\\
Oskar-Morgenstern-Platz 1\\ 1090 Wien\\ Austria}
\email{Gerald.Teschl@univie.ac.at}
\date{\today}
\begin{abstract}
The paper aims at developing the Riemann--Hilbert (RH) approach for the modified Camassa--Holm (mCH) equation on the line with non-zero boundary conditions, 
in the case when the solution is assumed to approach two different constants at
different sides of the line. We present detailed properties of spectral functions associated with 
the initial data for the Cauchy problem for the mCH equation and
obtain a representation for the solution of this problem in terms of the solution of an associated RH problem.
\end{abstract}
\maketitle
\section{Introduction}\label{sec:1}

%

In the present paper, we consider the initial value problem for the mCH equation \eqref{mCH-1}:
\begin{subequations}\label{mCH1-ic}
\begin{alignat}{4}           \label{mCH-1}
&m_t+\left((u^2-u_x^2)m\right)_x=0,&\quad&m\coloneqq u-u_{xx},&\quad&t>0,&\;&-\infty<x<+\infty,\\
&u(x,0)=u_0(x),&&&&&&-\infty<x<+\infty\label{IC},
\end{alignat}
\end{subequations}
assuming that 
\begin{equation}\label{mCH-bc}
u_0(x)\to\begin{cases}
A_1 \text{ as } x\to-\infty\\
A_2 \text{ as } x\to\infty\\
\end{cases},
\end{equation}
where $A_1$ and $A_2$ are some different constants, 
and that the solution $u(x,t)$ preserves this behavior for all fixed $t>0$.

Equation \eqref{mCH-1} is an integrable modification, with cubic nonlinearity, of the Camassa--Holm (CH) equation \cites{CH93,CHH94}
\begin{equation}\label{CH-1}
m_t+\left(u m\right)_x + u_x m=0,\quad m:= u-u_{xx}.
\end{equation}
The  Camassa--Holm equation has been studied intensively over more than two decades, due to its rich mathematical structure as well as applications for modeling the unidirectional propagation of shallow water waves over a flat bottom \cites{J02,CL09}. The CH and mCH equations are both integrable in the sense that they have Lax pair representations, which allows developing the inverse scattering transform (IST) method, in one form or another, to study the properties of solutions of initial (Cauchy) and initial boundary value problems for these equations. In particular, the inverse scattering method in the form of a Riemann--Hilbert (RH) problem developed for the CH equation with linear dispersion \cite{BS08} allowed studying the large-time behavior of solutions of initial as well as initial boundary value problems for the CH equation \cites{BS08-2,BKST09,BS09,BIS10} using the (appropriately adapted) nonlinear steepest descent method \cite{DZ93}.

Over the last few years various modifications and generalizations of the CH equation have been introduced, see, e.g., \cite{YQZ18} and references therein. Novikov \cite{N09} applied a perturbative symmetry approach in order to classify integrable equations of the form
\[
\left(1-\partial_x^2\right) u_t=F(u, u_x, u_{xx}, u_{xxx}, \dots),\qquad u=u(x,t), \quad \partial_x:=\partial/\partial x,
\]
assuming that $F$ is a homogeneous differential polynomial over $\D{C}$, quadratic or cubic in $u$ and its $x$-derivatives (see also \cite{MN02}). In the list of equations presented in \cite{N09}, equation (32), which was the second equation with \emph{cubic} nonlinearity, had the form \eqref{mCH-1}. In an equivalent form, this equation was given by Fokas in \cite{F95} (see also \cite{OR96} and \cite{Fu96}). Shiff \cite{S96} considered equation \eqref{mCH-1} as a dual to the modified Korteweg--de Vries (mKdV) equation and introduced a Lax pair for \eqref{mCH-1} by rescaling the entries of the spatial part of a Lax pair for the mKdV equation. An alternative (in fact, gauge equivalent) Lax pair for \eqref{mCH-1} was given by Qiao \cite{Q06}, so the mCH equation is also referred to as the Fokas--Olver--Rosenau--Qiao (FORQ) equation \cite{HFQ17}.

The local well-posedness and wave-breaking mechanisms for the mCH equation and its generalizations, particularly, the mCH equation with linear dispersion, are discussed in \cites{GLOQ13,FGLQ13,LOQZ14,CLQZ15,CGLQ16}. Algebro-geometric quasiperiodic solutions are studied in \cite{HFQ17}. The local well-posedness for classical solutions and for global weak solutions to \eqref{mCH-1} in Lagrangian coordinates are discussed in \cite{GL18}.

The Hamiltonian structure and Liouville integrability of peakon systems are discussed in \cites{AK18,OR96,GLOQ13,CS17}. In \cite{K16}, a Liouville-type transformation was presented relating the isospectral problems for the mKdV equation and  the mCH equation, and a Miura-type map from the mCH equation to the CH equation was introduced. The B\"acklund transformation for the mCH equation and a related nonlinear superposition formula are presented in \cite{WLM20}.

In the case of the Camassa--Holm equation, the inverse scattering transform method (particularly, in the form of a Riemann--Hilbert factorization problem) works for the version of this equation (considered for functions decaying at spatial infinity) that includes  an additional linear dispersion term.
Equivalently, this problem can be rewritten as a Cauchy problem for equation \eqref{CH-1} 
considered on a constant, nonzero background. Indeed, the inverse scattering transform method requires that the spatial equation from the Lax pair associated to the CH equation have continuous spectrum. On the other hand, the asymptotic analysis of the dispersionless CH equation \eqref{CH-1} on zero background (where the spectrum is purely discrete) requires a different tool (although having a certain analogy with the Riemann--Hilbert method), namely, the analysis of a coupling problem for entire functions \cites{E19,ET13,ET16}.

In the case of the mCH equation, the situation is similar: the inverse scattering method for the Cauchy problem can be developed when equation \eqref{mCH-1} is considered on a nonzero background. The Riemann--Hilbert formalism for this problem is developed in \cite{BKS20}, and the asymptotic analysis 
of the large-time behavior of the solutions on a uniform nonzero background is presented in 
\cite{BKS21}. 

Integrable nonlinear PDE  with non-vanishing boundary conditions at infinity has received plenty of
attention in the literature, see e.g.\cites{BilM19,BM17,Dem13,IU86}.
Particularly, initial value problems 
 with initial data approaching 
different ``backgrounds'' at different spatial infinities (the so-called step-like initial data) have attracted considerable attention because they can be used as   models for studying expanding, oscillatory dispersive shock waves (DSW),
which are   large scale, coherent excitation
in dispersive systems \cites{EH16, B18}. 
Large-time evolution of step-like initial data has been studied  for models of
 uni-directional (Korteweg–de Vries equation) wave propagation 
 \cites{A16, E13}
 as well as 
 bi-directional (Nonlinear Schr\"odinger equation) wave propagation
 \cites{BFP16,BLM21,BLS21,BLS22,BV07,FLQ21,J15}.

The RH problem formalism for the step-like initial value problem for the Camassa–Holm
equation was presented in \cite{M15}, and the large-time behavior of the
solutions of this problem was discussed in \cite{M16}.

In the present paper, we develop the Riemann--Hilbert formalism to problem \eqref{mCH1-ic}
with the step-like initial data \eqref{mCH-bc}
assuming that $0<A_1<A_2$ and that $u(x,t)$ approaches its large-$x$ limits sufficiently fast.
We also assume that $m(x,0)=u_0(x)-u_{0xx}(x)>0$ for all $x$; then it can be shown that  $m(x,t)>0$ for all $t$ (see Appendix \ref{app:A}, for the case of the CH equation, see \cites{C00,CE98-1}). 
In Section 2, we introduce appropriate transformations of the Lax pair equations
and the associated Jost solutions (``eigenfunctions'') and discuss analytic properties of the eigenfunctions and the corresponding spectral functions (scattering coefficients),
including the symmetries and the behavior at the branch points. Here the analysis is performed when fixing the branches of the functions
$k_j(\lambda):=\sqrt{\lambda^2-\frac{1}{A_j^2}}$, $j=1,2$ involved in the Lax
pair transformations as having the branch cuts $(-\infty, -\frac{1}{A_j})\cup(\frac{1}{A_j}, \infty)$.

 In Section 3, the introduced eigenfunctions 
 are used in the construction of the Riemann--Hilbert problems, whose
 solutions evaluated at $\lambda=0$ (where $\lambda$ is the spectral parameter
 in the Lax pair equations) give parametric representations of the 
 solution of problem \eqref{mCH1-ic}.

The case $0<A_2<A_1$ is briefly discussed in Appendix \ref{app:B}.

\begin{notations*}
In what follows, $\sigma_1\coloneqq\left(\begin{smallmatrix}0&1\\1&0\end{smallmatrix}\right)$, $\sigma_2\coloneqq\left(\begin{smallmatrix}0&-\ii\\\ii&0\end{smallmatrix}\right)$, and $\sigma_3\coloneqq\left(\begin{smallmatrix}1&0\\0&-1\end{smallmatrix}\right)$ denote the standard Pauli matrices, 
$\mathbb{C}^+:=\{\lambda\in\mathbb{C}|\Im(\lambda)> 0\}$, and
$\mathbb{C}^-:=\{\lambda\in\mathbb{C}|\Im(\lambda)< 0\}$.
\end{notations*}

\section{Lax pairs and eigenfunctions}\label{sec:2}
\subsection{Lax pairs}
The Lax pair for the mCH equation \eqref{mCH-1}  has the following form \cite{Q06}:

\begin{subequations}\label{Lax}
\begin{alignat}{4} \label{Lax-x}
    \Phi_x(x,t,\lambda)&=U(x,t,\lambda)\Phi(x,t,\lambda), \\
    \Phi_t(x,t,\lambda)&=V(x,t,\lambda)\Phi(x,t,\lambda), \label{Lax-t}
\end{alignat}
where  the coefficients $U$ and $V$ are  defined by
\begin{alignat}{4} \label{U}
U&=\frac{1}{2}\begin{pmatrix} -1 & \lambda  m \\
-\lambda m & 1\end{pmatrix},\\ \label{V}
V&=\begin{pmatrix}\lambda^{-2}+\frac{u^2-u_x^2}{2} &
-\lambda^{-1}(u-u_x)-\frac{\lambda(u^2-u_x^2)m}{2}\\
\lambda^{-1}(u+u_x)+\frac{\lambda(u^2- u_x^2)m}{2} &
		 -\lambda^{-2}-\frac{u^2-u_x^2}{2}\end{pmatrix},
\end{alignat}
\end{subequations}
with $m(x,t)= u(x,t)-u_{xx}(x,t)$. 
The RH formalism for integrable nonlinear equations is based on using appropriately defined eigenfunctions, i.e., solutions of the Lax pair, whose behavior as functions of the spectral parameter is well-controlled in the extended complex plane. Notice that the coefficient matrices $U$ and $V$ are traceless, which provides that the determinant of a matrix solution to \eqref{Lax} (composed of two vector solutions) is independent of $x$ and $t$.

Also notice that $U$ and $V$  have singularities (in the extended complex $\lambda$-plane) at $\lambda=0$ and $\lambda=\infty$. In particular, $U$ is singular at $\lambda=\infty$,
which necessitates a special care when constructing solutions with controlled behavior
as $\lambda\to\infty$. On the other hand, $U$ becomes $u$-independent at $\lambda=0$
(a property shared by many Camassa--Holm-typed equations, including the CH equation itself),
which suggests using the behavior of the constructed solutions as $\lambda\to 0$
in order to ``extract'' the solution of the nonlinear equation in question from
the solution of an associated Riemann--Hilbert problem (whose construction, in the direct problem,
involves the dedicated solutions of the Lax pair equations).

\subsubsection*{Notations}

\begin{itemize}
    \item 
We introduce the following notations for various intervals of the real axis:
\[
\Sigma_j=(-\infty,-\frac{1}{A_j}]\cup[\frac{1}{A_j},\infty), \qquad
\dot{\Sigma}_j=(-\infty,-\frac{1}{A_j})\cup(\frac{1}{A_j},\infty),
\]
\[
\Sigma_0=[-\frac{1}{A_1},-\frac{1}{A_2}]\cup[\frac{1}{A_2},\frac{1}{A_1}], \qquad
\dot{\Sigma}_0=(-\frac{1}{A_1},-\frac{1}{A_2})\cup(\frac{1}{A_2},\frac{1}{A_1}).
\]
Notice that $\Sigma_1\subset\Sigma_2$ since we assume $A_1<A_2$.
\item
For $\lambda\in \Sigma_j$ we denote by $\lambda_+$ ($\lambda_-$) the point of the upper (lower) side of $\Sigma_j$ (i.e. $\lambda_\pm=\lim_{\epsilon\downarrow 0}\lambda\pm\ii\epsilon$).
Then we have $-\lambda_+=(-\lambda)_-$ and $\overline{\lambda_+}=\lambda_-$.
\item
$k_j(\lambda):=\sqrt{\lambda^2-\frac{1}{A_j^2}}$, $j=1,2$ with the branch 
cut $\Sigma_j$ and the branch is fixed by the condition $k_j(0)=\frac{\ii}{A_j}$.
\end{itemize}

Observe that  $\Im k_j(\lambda)\geq 0$ on $\mathbb{C}$, and $k_j(\lambda)$ is real valued on the both sides of $\Sigma_j$.
Also notice that $k_j(\lambda)=\omega_j^+(\lambda)\omega_j^-(\lambda)$, where $\omega_j^+(\lambda)=\sqrt{\lambda-\frac{1}{A_j}}$ with the branch cut  $[\frac{1}{A_j},\infty)$ and $\omega_j^+(0)=\frac{\ii}{\sqrt{A_j}}$, and $\omega_j^-(\lambda)=\sqrt{\lambda+\frac{1}{A_j}}$ with the branch cut  $(-\infty,-\frac{1}{A_j}]$ and $\omega_j^-(0)=\frac{1}{\sqrt{A_j}}$.

Observe the following symmetry relations:
\begin{subequations}\label{sym_k_i}
\begin{alignat}{4}\label{sym_k_i-a}
k_j(-\lambda)&=k_j(\lambda),\quad \lambda\in\mathbb{C}\setminus\Sigma_j,\\ \label{sym_k_i-b}
k_j(\lambda_+)&=-k_j((-\lambda)_+),\quad\lambda\in\Sigma_j,\\\label{sym_k_i-c}
\overline{k_j(\overline{\lambda})}&=-k_j(\lambda),\quad\lambda\in\mathbb{C}\setminus\Sigma_j,\\ \label{sym_k_i-d}
\overline{k_j(\lambda_+)}&=k_j(\lambda_+),\quad\lambda\in\Sigma_j
\end{alignat}
\end{subequations}
(here  \eqref{sym_k_i-b} follows from \eqref{sym_k_i-a} and \eqref{sym_k_i-c}).

In order to control the large $\lambda$ behavior of solutions of \eqref{Lax}, we introduce two gauge transformations associated with $x\to (-1)^j\infty$ and $m\to A_j$ (in a similar way as it was done in the case of the constant background \cite{BKS20}).
\begin{proposition}
Equation \eqref{mCH-1} admits  Lax pairs  of the form ($j=1,2$)
\begin{subequations}\label{Lax-Q-form}
\begin{align}
\hat\Phi_{jx}+Q_{jx}\hat\Phi_j &= \hat U_j\hat\Phi_j,\\
\hat\Phi_{jt}+Q_{jt}\hat\Phi_j &= \hat V_j\hat\Phi_j,
\end{align}
\end{subequations}
whose coefficients $Q_j\equiv Q_j(x,t,\lambda)$, $\hat U_j\equiv\hat U_j(x,t,\lambda)$, and $\hat V_j\equiv\hat V_j(x,t,\lambda)$ are $2\times 2$ matrices 
given by \eqref{Lax-1} and \eqref{Qp}, which are characterized by the following properties:
\begin{enumerate}[\rm(i)]
\item
$Q_j$ is diagonal and is unbounded as $\lambda\to\infty$.
\item
$\hat U_j=\ord(1)$ and $\hat V_j=\ord(1)$ as $\lambda\to\infty$.
\item
The diagonal parts of $\hat U_j$ and $\hat V_j$ decay as $\lambda\to\infty$.
\item
$\hat U_j\to 0$ and $\hat V_j\to 0$ as $x\to(-1)^j\infty$.
\end{enumerate}
\end{proposition}

\begin{proof}
Notice that $U$ in \eqref{U} can be written as
\begin{equation}
\label{U-tilde}
U(x,t,\lambda)=\frac{ m(x,t)}{2A_j}\begin{pmatrix}-1&\lambda A_j\\-\lambda A_j&1\end{pmatrix}+\frac{ m(x,t)-A_j}{2A_j}\begin{pmatrix}1&0\\
0&-1\end{pmatrix},
\end{equation}
where $m(x,t)-A_j\to 0$ as $x\to (-1)^j\infty$. The  first (non-decaying, as $x\to(-1)^j\infty$) term in \eqref{U-tilde} can be diagonalized by introducing
\begin{equation}
\label{hatPhi-Phi}
\hat\Phi_j(x,t,\lambda)\coloneqq D_j(\lambda)\Phi(x,t,\lambda),
\end{equation}
where
\begin{equation}
\label{D-j}
D_j(\lambda)\coloneqq \sqrt{\frac{1}{2}}\sqrt{\frac{1}{\ii A_jk_j(\lambda)}-1}\begin{pmatrix}
\frac{\lambda A_j}{1-\ii A_jk_j(\lambda)} & - 1  \\
- 1 & \frac{\lambda A_j}{1-\ii A_jk_j(\lambda)} \\
\end{pmatrix}
\end{equation}
with
\[
D^{-1}_j(\lambda)\coloneqq \sqrt{\frac{1}{2}}\sqrt{\frac{1}{\ii A_jk_j(\lambda)}-1}\begin{pmatrix}
\frac{\lambda A_j}{1-\ii A_jk_j(\lambda)} & 1  \\
1 & \frac{\lambda A_j}{1-\ii A_jk_j(\lambda)} \\
\end{pmatrix}.
\]
The factor $\sqrt{\frac{1}{2}}\sqrt{\frac{1}{\ii A_jk_j(\lambda)}-1}$ provides $\det D_j(\lambda)=1$ for all $\lambda$, and 
the branch of the square root is chosen so that the branch cut is $[0,\infty)$ and $\sqrt{-1}=\ii$;
then $\overline{\sqrt{w_j}}=-\sqrt{\overline{w_j}}$.
Observe that $\sqrt{\frac{1}{\ii A_jk_j(\lambda)}-1}$ is well defined as a function of $\lambda$ on $\mathbb{C}\setminus\Sigma_j$ as well as on the sides of $\Sigma_j$.
Then \eqref{hatPhi-Phi} transforms \eqref{Lax-x} into
\begin{subequations}\label{Lax-1}
\begin{equation}\label{Lax-1-x}
\hat\Phi_{jx}+\frac{\ii k_j(\lambda)m}{2}\sigma_3\hat\Phi_j=\hat U_j \hat\Phi_j,
\end{equation}
where $\hat U_j\equiv\hat U_j(x,t,\lambda)$ is given by
\begin{equation}\label{U_i-hat}
\hat U_j=\frac{\lambda(m-A_j)}{2 A_j k_j(\lambda)}
\sigma_2
+\frac{m-A_j}{2\ii A_j^2k_j(\lambda)}\sigma_3.
\end{equation}
In turn, the $t$-equation \eqref{Lax-t} of the Lax pair is transformed into
\begin{equation}\label{Lax-1-t}
\hat\Phi_{jt} +\ii A_j k_j(\lambda)
\left(-\frac{1}{2A_j} m(u^2-u_x^2)-\frac{1}{\lambda^2}\right)\sigma_3\hat\Phi_j= \hat V_j \hat\Phi_j,
\end{equation}
where $\hat V_j\equiv\hat V_j(x,t,\lambda)$ is given by
\begin{equation}\label{hat-V_i}
\begin{aligned}
\hat V_j&=
-\frac{1}{2  A_j k_j(\lambda)}\left(\lambda (u^2-u_x^2)( m - A_j)
+\frac{2 (u-A_j)}{\lambda}\right)
\sigma_2
+\frac{ \tilde u_x}{\lambda}  \sigma_1 \\
&\quad-\frac{1}{\ii A_j k_j(\lambda)}\left(A_j( u-A_j) +\frac{1}{2A_j}(u^2-u_x^2)( m - A_j)\right) \sigma_3.
\end{aligned}
\end{equation}
\end{subequations}

Now notice that equations \eqref{Lax-1-x} and \eqref{Lax-1-t} have the desired form \eqref{Lax-Q-form}, if we define $Q_j$ by
\begin{subequations}\label{Qp}
\begin{equation}\label{Q}
Q_j(x,t,\lambda)\coloneqq p_j(x,t,\lambda)\sigma_3,
\end{equation}
with
\begin{equation}\label{p_i}
p_j(x,t,\lambda)\coloneqq \ii A_j k_j(\lambda)\left(\frac{1}{2A_j}\int^x_{(-1)^j\infty} ( m(\xi,t)-A_j)\dd\xi+\frac{x}{2}-t\big(\frac{1}{\lambda^2}+\frac{A_j^2}{2}\big)\right).
\end{equation}
\end{subequations}
Indeed, we obviously have $p_{jx}=\frac{\ii k_j(\lambda)m}{2}$; on the other hand, 
the equality 
\[
p_{jt}=\ii A_j k_j(\lambda)
\left(-\frac{1}{2A_j} m(u^2-u_x^2)-\frac{1}{\lambda^2}\right)
\]
follows from 
 \eqref{mCH-1}.
\end{proof}
\begin{remark}
In \cite{BKS20}, which deals with the mCH equation on a single background,
introducing a uniformizing spectral parameter (such that $\lambda$ and the respective $k(\lambda)$
are rational with respect to it) allowed getting rid of square roots and thus avoiding
the problem of specifying particular branches.
In the present case, since we have to deal with two different functions, $k_1(\lambda)$
and $k_2(\lambda)$, associated with two different backgrounds, we keep the original
spectral parameter $\lambda$ as the spectral variable in the  RH problem formalism we are going to develop.
\end{remark}
\subsection{Eigenfunctions}

The Lax pair in the form \eqref{Lax-1} allows us to determine, via associated integral equations, dedicated solutions having a well-controlled behavior as functions of the spectral parameter $\lambda$ for large values of $\lambda$. Indeed, introducing
\begin{equation}\label{zam}
\widetilde\Phi_j=\hat\Phi_j\eul^{Q_j}
\end{equation}
(understanding $\widetilde\Phi_j$ as a $2\times 2$ matrix), equations \eqref{Lax-1-x} and \eqref{Lax-1-t} can be rewritten as
\begin{equation}\label{comsys}
\begin{cases}
\widetilde\Phi_{jx}+[Q_{jx},\widetilde\Phi_j]=\hat U_j\widetilde\Phi_j,&\\
\widetilde\Phi_{jt}+[Q_{jt},\widetilde\Phi_j]=\hat V_j\widetilde\Phi_j,&
\end{cases}
\end{equation}
where $[\,\cdot\,,\,\cdot\,]$ stands for the commutator. We now determine the Jost solutions $\widetilde\Phi_{j}\equiv\widetilde\Phi_{j}(x,t,\lambda)$, $j=1,2$ of \eqref{comsys} as the solutions of the associated Volterra integral equations:
\begin{equation}\label{inteq}
\widetilde\Phi_{j}(x,t,\lambda)=I+\int_{(-1)^j\infty}^x
	\eul^{Q_j(\xi,t,\lambda)-Q_j(x,t,\lambda)}\hat U_j(\xi,t,\lambda)\widetilde\Phi_{j}(\xi,t,\lambda)
		\eul^{Q_j(x,t,\lambda)-Q_j(\xi,t,\lambda)}\dd\xi,
\end{equation}
or, taking into account the definition \eqref{Qp} of $Q_j$,
\begin{equation}\label{eq}
\widetilde\Phi_{j}(x,t,\lambda)=I+\int_{(-1)^j\infty}^x
	\eul^{\frac{\ii k_j(\lambda)}{2}\int^\xi_x m(\tau,t)\dd\tau\sigma_3}\hat U_j(\xi,t,\lambda)\widetilde\Phi_j(\xi,t,\lambda)
			\eul^{-\frac{\ii k_j(\lambda)}{2}\int^\xi_x m(\tau,t)\dd\tau\sigma_3}\dd\xi,
\end{equation}
($I$ is the $2\times 2$ identity matrix). 

Hereafter,  $\hat\Phi_j\coloneqq\widetilde\Phi_j\eul^{-Q_j}$, $j=1,2$ denote the corresponding Jost solutions of \eqref{Lax-1} whereas  $\Phi_j\coloneqq D_j^{-1}(\lambda)\hat\Phi_j$ denote the corresponding Jost solutions of \eqref{Lax}.

We are now able to analyze the analytic and asymptotic properties of the solutions $\widetilde\Phi_j$ of \eqref{eq} as functions of $\lambda$, using Neumann series expansions. Let $A^{(1)}$ and $A^{(2)}$ denote the columns of a $2\times 2$ matrix $A=\left(A^{(1)}\ \ A^{(2)}\right)$. Using these notations we have the following properties:

\begin{enumerate}[\textbullet]
\item
$\widetilde\Phi_j^{(j)}$ is analytic in $\mathbb{C}\setminus \Sigma_j$ and has a continuous extension on the lower and upper sides of $\dot{\Sigma}_j$.
\item
$\widetilde\Phi_j^{(1)}$ and $\widetilde\Phi_j^{(2)}$ are well defined and continuous on the lower and upper sides of $\dot{\Sigma}_j$.

\end{enumerate}

In \eqref{comsys} the coefficients are traceless matrices, from which it follows that $\det \tilde\Phi_j$ is independent on $x$ and $t$, and hence
\begin{enumerate}[\textbullet]
\item
$\det\widetilde\Phi_j\equiv 1$.
\end{enumerate}

Regarding the values of $\widetilde\Phi_j$ at particular points in the $\lambda$-plane, \eqref{eq} implies the following:
\begin{enumerate}[\textbullet]
\item
$\left(\begin{smallmatrix}
\widetilde\Phi_1^{(1)} &
\widetilde\Phi_2^{(2)}\end{smallmatrix}\right)\to I$ as $\lambda\to\infty$ (since the diagonal part of $\hat U_j$ is $\ord(\frac{1}{\lambda})$ and the off-diagonal part of $\hat U_j$ is bounded).
\item $\tilde\Phi_j$ has singularities at  $\lambda=\pm\frac{1}{A_j}$  of order $\frac{1}{2}$ (this will be discussed below, see Subsection \ref{sec:branch points}). 
\end{enumerate}
\subsection{``Background'' solution}\label{sec:Background-solution}

Introduce $\Phi_{0,j}(x,t,\lambda):=D_j^{-1}(\lambda)\eul^{-Q_j(x,t,\lambda)}$. We see 
that $\Phi_{0,j}$ satisfy the differential equations:

\begin{equation}\label{back-sol}
\begin{cases}
\Phi_{0,jx}=\frac{ m(x,t)}{2A_j}\begin{pmatrix}-1&\lambda A_j\\-\lambda A_j&1\end{pmatrix}\Phi_{0,j},&\\
\Phi_{0,jt}=\left(-\frac{1}{2A_j} m(u^2-u_x^2)-\frac{1}{\lambda^2}\right)\begin{pmatrix}-1&\lambda A_j\\-\lambda A_j&1\end{pmatrix}\Phi_{0,j}.&
\end{cases}
\end{equation}

Comparing this with \eqref{Lax-Q-form}, $\Phi_{j}(x,t,\lambda)$ can be characterized as the solutions of the integral equations:
\begin{equation}\label{eq_phi}
\Phi_{j}(x,t,\lambda)=\Phi_{0,j}(x,t,\lambda)+\int_{(-1)^j\infty}^x
	\Phi_{0,j}(x,t,\lambda)\Phi_{0,j}^{-1}(y,t,\lambda)\frac{m(y,t)-A_j}{2A_j} \sigma_3\Phi_{j}(y,t,\lambda)dy.
\end{equation}

Observe that $\Phi_{0,j}(x,t,\lambda)\Phi_{0,j}^{-1}(y,t,\lambda)$ is entire w.r.t. $\lambda$. Hence the ``lack of analyticity'' (jumps) of $\Phi_{j}(x,t,\lambda)$ is generated by  the   ``lack of analyticity'' of $\Phi_{0,j}(x,t,\lambda)$. Notice that  $\det \Phi_j = \det \Phi_{0,j}=1$.


\subsection{Spectral functions}  \label{sec:spectral-data}

Introduce the scattering matrices $s(\lambda_\pm)$ for $\lambda\in\dot{\Sigma}_1$
as  matrices relating $\Phi_1$ and $\Phi_2$:
\begin{equation}\label{scat}
\Phi_1(x,t,\lambda_\pm)=\Phi_2(x,t,\lambda_\pm)s(\lambda_\pm),\qquad\lambda\in\dot{\Sigma}_1
\end{equation}
with $\det s(\lambda_\pm)=1$. In turn,
 $\tilde\Phi_1$ and $\tilde\Phi_2$ are related by
\begin{equation}\label{scat_}
D_1^{-1}(\lambda_\pm)\tilde\Phi_1(x,t,\lambda_\pm)=D_2^{-1}(\lambda_\pm)\tilde\Phi_2(x,t,\lambda_\pm)\eul^{-Q_2(x,t,\lambda_\pm)}s(\lambda_\pm)\eul^{Q_1(x,t,\lambda_\pm)},~\lambda\in\dot{\Sigma}_1.
\end{equation}
Introducing
\begin{equation}\label{scat_til}
\tilde s(x,t,\lambda_\pm):=\eul^{-Q_2(x,t,\lambda_\pm)}s(\lambda_\pm)\eul^{Q_1(x,t,\lambda_\pm)}
\end{equation}
 we have 
\begin{equation}\label{scat_2}
(D_1^{-1}\tilde\Phi_1)(x,t,\lambda_\pm)=(D_2^{-1}\tilde\Phi_2)(x,t,\lambda_\pm)\tilde s(x,t,\lambda_\pm), \qquad\lambda\in\dot{\Sigma}_1.
\end{equation}

Notice that the scattering coefficients ($s_{ij}$) can be expressed as follows:
\begin{subequations}\label{scatcoeff}
\begin{alignat}{3}\label{scatcoeff-a}
s_{11}&=\det(\Phi_1^{(1)},\Phi_2^{(2)}),\\\label{scatcoeff-b}
s_{12}&=\det(\Phi_1^{(2)},\Phi_2^{(2)}),\\\label{scatcoeff-c}
s_{21}&=\det(\Phi_2^{(1)},\Phi_1^{(1)}),\\\label{scatcoeff-d}
s_{22}&=\det(\Phi_2^{(1)},\Phi_1^{(2)}).
\end{alignat}
\end{subequations}
Accordingly, 
\begin{subequations}\label{til_scatcoeff}
\begin{alignat}{3}\label{til_scatcoeff-a}
\tilde s_{1j}&=\det((D_1^{-1}\tilde\Phi_1)^{(j)},(D_2^{-1}\tilde\Phi_2)^{(2)}),\\\label{til_scatcoeff-c}
\tilde s_{2j}&=\det((D_2^{-1}\tilde\Phi_2)^{(1)},(D_1^{-1}\tilde\Phi_1)^{(j)}).
\end{alignat}
\end{subequations}

Then \eqref{scatcoeff-a} implies that  $s_{11}(\lambda)$ can be analytically extended to $\mathbb{C}\setminus\Sigma_2$ and defined on the upper and lower sides of $\dot\Sigma_2$. On the other hand, since $\Phi_1^{(1)}$ is analytic in $\mathbb{C}\setminus\Sigma_1$ and $\Phi_2^{(1)}$ is defined on the upper and lower sides of $\Sigma_2$, $s_{21}(\lambda)$ can be extended by \eqref{scatcoeff-c} to the lower and upper sides of $\dot{\Sigma}_2$. It follows that 
\eqref{scat} and \eqref{scat_} restricted to the first column   hold also on $\Sigma_0$, namely,
\begin{equation}\label{scat-col}
    \Phi_1^{(1)}(x,t,\lambda_\pm)=s_{11}(\lambda_\pm)\Phi_2^{(1)}(x,t,\lambda_\pm)+s_{21}(\lambda_\pm)\Phi_2^{(2)}(x,t,\lambda_\pm), \quad\lambda\in\dot\Sigma_0,
\end{equation}
and, respectively,
\begin{equation}\label{scat-col_}
    (D_1^{-1}\tilde\Phi_1^{(1)})(\lambda_\pm)=\tilde s_{11}(\lambda_\pm)(D_2^{-1}\tilde\Phi_2^{(1)})(\lambda_\pm)+\tilde s_{21}(\lambda_\pm)(D_2^{-1}\tilde\Phi_2^{(2)})(\lambda_\pm), \quad\lambda\in\dot\Sigma_0.
\end{equation}
\subsection{Symmetries}  \label{sec:symmetries}
Let's analyse the  symmetry relations amongst the eigenfunctions and scattering coefficients.
In order to simplify the notations, we will omit the dependence on $x$ and $t$ (e.g., $U(\lambda)\equiv U(x,t,\lambda)$).

\textit{First symmetry: $\lambda \longleftrightarrow -\lambda$.}

\begin{proposition}\label{prop:sym_Phi_minus_notin_sigma_i} The following symmetries hold:
    \begin{subequations}\label{sym_Phi_minus_notSigma_2}
    \begin{alignat}{3}\label{sym_Phi_minus_notSigma_2-a}
    \Phi_1^{(1)}(\lambda)&=-\sigma_3\Phi_1^{(1)}(-\lambda),\quad \lambda\in\mathbb{C}\setminus\Sigma_1,\\ \label{sym_Phi_minus_notSigma_2-b}
    \Phi_2^{(2)}(\lambda)&=\sigma_3\Phi_2^{(2)}(-\lambda),\quad \lambda\in\mathbb{C}\setminus\Sigma_2.
    \end{alignat}
    \end{subequations}
\end{proposition}

\begin{proof}
Observe that $\sigma_3U(\lambda)\sigma_3\equiv U(-\lambda)$ and $\sigma_3V(\lambda)\sigma_3\equiv V(-\lambda)$. Hence $\sigma_3 \Phi_j^{(j)}(-\lambda)$ solves \eqref{Lax} together with $\Phi_j^{(j)}(\lambda)$. Comparing their asymptotic behaviour as $x\to (-1)^j\infty$ and using \eqref{sym_k_i-a}, the symmetries \eqref{sym_Phi_minus_notSigma_2} follow.
\end{proof}

\begin{corollary}
We have
\begin{enumerate}
    \item \begin{equation}\label{sym_s11_minus_notSigma_2}
    s_{11}(-\lambda)=s_{11}(\lambda), \quad  \lambda\in\mathbb{C}\setminus\Sigma_2.  
    \end{equation}
    
    \item 
    \begin{subequations}\label{sym_tildePhi_minus_notSigma_2}
    \begin{alignat}{3}
    \tilde \Phi_1^{(1)}(\lambda)=\sigma_3\tilde         \Phi_1^{(1)}(-\lambda),\quad \lambda\in\mathbb{C}\setminus\Sigma_1, \\
    \tilde \Phi_2^{(2)}(\lambda)=-\sigma_3\tilde     \Phi_2^{(2)}(-\lambda),\quad \lambda\in\mathbb{C}\setminus\Sigma_2.
    \end{alignat}
    \end{subequations}
    
    \item 
\begin{subequations}\label{sym_minus_} 
\begin{alignat}{4}\label{sym_minus_a} 
(D_1^{-1}\tilde\Phi_1^{(1)})(-\lambda)&=-\sigma_3(D_1^{-1}\tilde\Phi_1^{(1)})(\lambda), \quad \lambda\in\mathbb{C}\setminus\Sigma_1,\\ \label{sym_minus_b} 
(D_2^{-1}\tilde\Phi_2^{(2)})(-\lambda)&=\sigma_3(D_2^{-1}\tilde\Phi_2^{(2)})(\lambda), \quad \lambda\in\mathbb{C}\setminus\Sigma_2.
\end{alignat}
\end{subequations}
\textcolor{red}{
}
\end{enumerate}
\end{corollary}

\begin{proof} \begin{enumerate}
    \item Substitute \eqref{sym_Phi_minus_notSigma_2} into \eqref{scatcoeff-a}.
    
    \item Observe that due to \eqref{sym_k_i-a}, we have $D_j^{-1}(-\lambda)=-\sigma_3D_j^{-1}(\lambda)\sigma_3$ and $Q_j(-\lambda)=Q_j(\lambda)$. Combining this with \eqref{sym_Phi_minus_notSigma_2} and using the connection between $\Phi_j$ and $\tilde\Phi_j$, we obtain \eqref{sym_tildePhi_minus_notSigma_2}. 
    
    \item Combine $D_j^{-1}(-\lambda)=-\sigma_3D_j^{-1}(\lambda)\sigma_3$ and \eqref{sym_tildePhi_minus_notSigma_2}.
\end{enumerate}
\end{proof}

\begin{proposition}\label{prop:sym_Phi_minus_sigma_i} The following symmetry holds
    \begin{equation}\label{sym_Phi_minus_sigma_i}
    \Phi_j(\lambda_+)=-\sigma_3\Phi_j(-\lambda_+)\sigma_3, \qquad \lambda\in\dot\Sigma_j.
\end{equation}
\end{proposition}

\begin{proof}
Since $\sigma_3U(\lambda)\sigma_3\equiv U(-\lambda)$ and $\sigma_3V(\lambda)\sigma_3\equiv V(-\lambda)$ and $U$ and $V$ do not have jumps along $\Sigma_j$, it follows that if $\Phi_j(\lambda_+)$ solves \eqref{Lax}, so does $\sigma_3 \Phi_j(-\lambda_+)$. Comparing their asymptotic behaviour as $x\to(-1)^j\infty$ and using \eqref{sym_k_i-a}, the symmetry \eqref{sym_Phi_minus_sigma_i} follows.
\end{proof}

\begin{corollary} We have
\begin{enumerate}
        \item       \begin{equation}\label{sym_s_minus_sigma1} s(\lambda_+)=\sigma_3s(-\lambda_+)\sigma_3,\quad \lambda\in\dot\Sigma_1  
        \end{equation}
        \item         \begin{equation}\label{sym_tildePhi_minus_sigma_i}
        \tilde\Phi_j(\lambda_+)=\sigma_3\tilde\Phi_j(-\lambda_+)\sigma_3, \qquad \lambda\in\dot\Sigma_j.
        \end{equation}
        
        \item 
\begin{equation}\label{sym_minus} 
 (D_j^{-1}\tilde\Phi_j)((-\lambda)_-)=-\sigma_3(D_j^{-1}\tilde\Phi_j)(\lambda_+)\sigma_3, \quad \lambda_+\in\dot\Sigma_j.
 \end{equation}
\textcolor{red}{
}

\end{enumerate}
\end{corollary}

\begin{proof} \begin{enumerate}
    \item Substitute \eqref{sym_Phi_minus_sigma_i} into \eqref{scat}.
    
    \item Observe that due to \eqref{sym_k_i-a}, we have $D_j^{-1}(-\lambda_+)=-\sigma_3D_j^{-1}(\lambda_+)\sigma_3$ and $Q_j(-\lambda_+)=Q_j(\lambda_+)$. Combining this with \eqref{sym_Phi_minus_sigma_i} and using the connection between $\Phi_j$ and $\tilde\Phi_j$, we obtain \eqref{sym_tildePhi_minus_sigma_i}.
    
    \item Combine $D_j^{-1}(-\lambda_+)=-\sigma_3D_j^{-1}(\lambda_+)\sigma_3$ and \eqref{sym_tildePhi_minus_sigma_i}.
\end{enumerate}
\end{proof}

\textit{Second symmetry: $\lambda \longleftrightarrow -\overline\lambda$.}

\begin{proposition}\label{prop:sym-Phi-(minus)} The following symmetry holds
\begin{equation}\label{sym_Phi_(minus)}
        \Phi_j(\lambda_+)=\sigma_3\Phi_j((-\lambda)_+)\sigma_2, \qquad \lambda\in\dot\Sigma_j.
\end{equation}
\end{proposition}

\begin{proof}
Since $U$ and $V$ are single valued functions of $\lambda$, we have  $\sigma_3U(\lambda_+)\sigma_3\equiv U((-\lambda)_+)$ and $\sigma_3V(\lambda_+)\sigma_3\equiv V((-\lambda)_+)$ for $\lambda\in\Sigma_j$. Hence, if $\Phi_j(\lambda_+)$ solves \eqref{Lax}, so does $\sigma_3 \Phi_j((-\lambda_+)$. Comparing their asymptotic behaviour as $x\to(-1)^j\infty$ and using \eqref{sym_k_i-b} and the equality $\sqrt{\frac{1}{\ii A_jk_j(\lambda_+)}-1}\sqrt{-\frac{1}{\ii A_jk_j(\lambda_+)}-1}=-\frac{\lambda_+}{k_j(\lambda_+)}$ for $\lambda_+\in\dot\Sigma_j$, the symmetry \eqref{sym_Phi_(minus)} follows.
\end{proof}


\begin{corollary} We have
\begin{enumerate}
    \item 
 \begin{equation}\label{sym_s_(minus)_sigma1}
     s(\lambda_+)=\sigma_2s((-\lambda)_+)\sigma_2,\qquad\lambda\in\dot\Sigma_1.
 \end{equation}
    \item 
    \begin{equation}\label{sym_s_(minus)_sigma1_2}
    s(\lambda_+)=\sigma_1s(\lambda_-)\sigma_1,\qquad \lambda\in \dot\Sigma_1.   
    \end{equation}

\item 
\begin{equation}\label{sym_tildePhi_(minus)_sigma_i}
    \tilde\Phi_j(\lambda_+)=\sigma_2\tilde\Phi_j((-\lambda)_+)\sigma_2, \qquad \lambda\in\dot\Sigma_j.
\end{equation}

\item 
 \begin{equation}\label{sym_(minus)} 
 (D_j^{-1}\tilde\Phi_j)((-\lambda)_+)=\sigma_3(D_j^{-1}\tilde\Phi_j)(\lambda_+)\sigma_2, \qquad \lambda\in\dot\Sigma_j.
\end{equation} 
\end{enumerate}
\end{corollary} 

\begin{proof}
\begin{enumerate}
    \item Substitute \eqref{sym_Phi_(minus)} into \eqref{scat}.
    
    \item Combine \eqref{sym_s_(minus)_sigma1} with \eqref{sym_s_minus_sigma1}.
    
    \item Observe that $k_j(\lambda_+)\in\mathbb{R}$ and that due to \eqref{sym_k_i-b} and $\sqrt{\frac{1}{\ii A_jk_j(\lambda_+)}-1}\sqrt{-\frac{1}{\ii A_jk_j(\lambda_+)}-1}=-\frac{\lambda_+}{k_j(\lambda_+)}$, we have $D_j(\lambda_+)\sigma_3D_j^{-1}((-\lambda)_+)=\sigma_2$ and $Q_j((-\lambda)_+)=-Q_j(\lambda_+)$ for $\lambda\in\dot\Sigma_j$. Combining this with \eqref{sym_Phi_(minus)} and using the connection between $\Phi_j$ and $\tilde\Phi_j$, we obtain \eqref{sym_tildePhi_(minus)_sigma_i}.
    
    \item Combine $D_j(\lambda_+)\sigma_3D_j^{-1}((-\lambda)_+)=\sigma_2$ and \eqref{sym_tildePhi_(minus)_sigma_i}.
    
\end{enumerate}
\end{proof}

\textit{Third symmetry: $\lambda \longleftrightarrow \overline\lambda$.} 

\begin{proposition}\label{prop:sym-Phi-barbar_notin_sigma_i} The following symmetries hold
     \begin{equation}\label{sym_phi_barbar-notin_Sigma_i}
     \overline{\Phi_j^{(j)}(\overline{\lambda})}=-\Phi_j^{(j)}(\lambda), \quad \lambda\in\mathbb{C}\setminus\Sigma_j.    
     \end{equation}
\end{proposition}

\begin{proof} Since $\overline{U(\overline\lambda)}\equiv U(\lambda)$ and $\overline{V(\overline\lambda)}\equiv V(\lambda)$, it follows that  $\overline{\Phi_j^{(j)}(\overline\lambda)}$ solves \eqref{Lax-x} together with $\Phi_j^{(j)}(\lambda)$. Hence, comparing their asymptotic behaviour as $x\to (-1)^j\infty$ and using \eqref{sym_k_i-c} and the equality $\overline{\sqrt{\frac{1}{\ii A_j k_j(\overline{\lambda})}-1}}=-\sqrt{\frac{1}{\ii A_j k_j(\lambda)}-1}$, we obtain the symmetries \eqref{sym_phi_barbar-notin_Sigma_i}.
\end{proof}

\begin{corollary}
We have
\begin{enumerate}
    \item 
    \begin{equation}\label{sym_s11_barbar}
    \overline{s_{11}(\overline{\lambda})}=s_{11}(\lambda),\quad\lambda\in\mathbb{C}\setminus\Sigma_2.  
    \end{equation}
    
    \item 

    \begin{equation}\label{sym_tildephi_barbar_notSigma_i}
    \overline{\tilde\Phi_j^{(j)}(\overline{\lambda})}=\tilde\Phi_j^{(j)}(\lambda),\quad \lambda\in\mathbb{C}\setminus\Sigma_j.    
    \end{equation}

    \item 
    \begin{equation}\label{sym_barbar_notin_Sigma_i}
    \overline{(D_j^{-1}\tilde\Phi_j^{(j)})(\overline{\lambda})}=-(D_j^{-1}\tilde\Phi_j^{(j)})(\lambda), \quad \lambda\in\mathbb{C}\setminus\Sigma_j.
\end{equation}
\end{enumerate}
\end{corollary}

\begin{proof} \begin{enumerate}
    \item Substitute \eqref{sym_phi_barbar-notin_Sigma_i} into \eqref{scatcoeff-a}.
    
    \item Observe that due to \eqref{sym_k_i-c} and $\overline{\sqrt{\frac{1}{\ii A_jk_j(\overline{\lambda})}-1}}=-\sqrt{\frac{1}{\ii A_jk_j(\lambda)}-1}$, we have $\overline{D_j^{-1}(\overline{\lambda})}=- D_j^{-1}(\lambda)$ and $\overline{Q_j(\overline{\lambda})}=Q_j(\lambda)$. Hence combining this with \eqref{sym_phi_barbar-notin_Sigma_i} and using the connection between $\Phi_j$ and $\tilde\Phi_j$, we obtain \eqref{sym_tildephi_barbar_notSigma_i}. 
    
    \item Combine $\overline{D_j^{-1}(\overline{\lambda})}=- D_j^{-1}(\lambda)$ and \eqref{sym_tildephi_barbar_notSigma_i}.
\end{enumerate}
\end{proof}

\begin{proposition}\label{prop:sym-Phi-barbar_Sigma_i} The following symmetry holds
\begin{equation}\label{sym_Phi_barbar_Sigma_i}
    \overline{\Phi_j(\overline{\lambda_+})}=-\Phi_j(\lambda_+),  \qquad \lambda\in\dot\Sigma_j.
\end{equation}
\end{proposition}

\begin{proof}
As above, since $\overline{U(\overline\lambda)}\equiv U(\lambda)$ and $\overline{V(\overline\lambda)}\equiv V(\lambda)$ and $U$ and $V$ have no jumps along $\Sigma_j$, we have $\overline{U(\lambda_-)}\equiv U(\lambda_+)$ and $\overline{V(\lambda_-)}\equiv V(\lambda_+)$. It follows that  if $\Phi_j(\lambda_+)$ solves \eqref{Lax}, so does $\overline{ \Phi_j(\overline\lambda_+)}$. Comparing their asymptotic behaviour as $x\to(-1)^j\infty$ and using \eqref{sym_k_i-c} and the fact that $\overline{\sqrt{\frac{1}{\ii A_jk_j(\overline{\lambda})}-1}}=-\sqrt{\frac{1}{\ii A_jk_j(\lambda)}-1}$, the symmetry \eqref{sym_Phi_barbar_Sigma_i} follows.
\end{proof}

\begin{corollary} We have
\begin{enumerate}
    \item 
    \begin{equation}\label{sym_s_barbar}
        \overline{s(\overline{\lambda_+})}=s(\lambda_+),\quad \lambda\in\dot\Sigma_1.
    \end{equation}
    
    \item 
    \begin{equation}\label{sym_tildePhi_barbar_Sigma_i}
    \overline{\tilde\Phi_j(\overline{\lambda_+})}=\tilde\Phi_j(\lambda_+), \qquad \lambda\in\dot\Sigma_j.
    \end{equation}
    
    \item 
\begin{equation}\label{sym_bar_Sigma_i}  \overline{(D_j^{-1}\tilde\Phi_j)(\overline{\lambda_+})}=-(D_j^{-1}\tilde\Phi_j)(\lambda_+), \quad \lambda\in\dot\Sigma_j.  
\end{equation}
\end{enumerate}
\end{corollary} 

\begin{proof} \begin{enumerate}
    \item Substitute \eqref{sym_Phi_barbar_Sigma_i} into \eqref{scat}.
    
    \item Observe that due to \eqref{sym_k_i-c} and $\overline{\sqrt{\frac{1}{\ii A_jk_j(\overline{\lambda})}-1}}=-\sqrt{\frac{1}{\ii A_jk_j(\lambda)}-1}$, we have $\overline{D_j^{-1}(\lambda_-)}=- D_j^{-1}(\lambda_+)$ and $\overline{Q_j(\lambda_-)}=Q_j(\lambda_+)$ for $\lambda\in\dot\Sigma_j$.  Combining this with \eqref{sym_Phi_barbar_Sigma_i} and using the connection between $\Phi_j$ and $\tilde\Phi_j$, we obtain the result.
    
    \item Combine $\overline{D_j^{-1}(\lambda_-)}=- D_j^{-1}(\lambda_+)$ and $\overline{Q_j(\lambda_-)}=Q_j(\lambda_+)$ and \eqref{sym_tildePhi_barbar_Sigma_i}.
\end{enumerate}
\end{proof}

\textit{Fourth symmetry $\lambda_+ \longleftrightarrow \lambda_+$.}

\begin{proposition}\label{prop:sym-Phi-bar} The following symmetry holds
\begin{equation}\label{sym_Phi_bar}
    \overline{\Phi_j(\lambda_+)}=\ii \Phi_j(\lambda_+)\sigma_1, \qquad \lambda\in\dot\Sigma_j.
    \end{equation}
\end{proposition}

\begin{proof}
Since $\overline{U(\lambda_+)}\equiv U(\lambda_+)$ and $\overline{V(\lambda_+)}\equiv V(\lambda_+)$ for $\lambda\in\Sigma_j$, 
in follows that if $\Phi_j(\lambda_+)$ solves \eqref{Lax}, so does $\overline{\Phi_j(\lambda_+)}$. Comparing their asymptotic behaviour as $x\to(-1)^j\infty$ and using \eqref{sym_k_i-d} and the equalities  $\sqrt{-\frac{1}{\ii A_jk_j(\lambda_+)}-1}\cdot\frac{\lambda_+A_j}{1+\ii A_jk_j(\lambda_+)}=-\ii\sqrt{\frac{1}{\ii A_jk_j(\lambda_+)}-1}$ and $\sqrt{\frac{1}{\ii A_jk_j(\lambda_+)}-1}\cdot\frac{\lambda_+A_j}{1-\ii A_jk_j(\lambda_+)}=\ii\sqrt{-\frac{1}{\ii A_jk_j(\lambda_+)}-1}$  for $\lambda\in\dot\Sigma_j$, the symmetry \eqref{sym_Phi_bar} follows.
\end{proof}

\begin{corollary}\label{cor: sym} We have
\begin{enumerate}
    
    \item $s(\lambda_+)=\sigma_1\overline{s(\lambda_+)}\sigma_1, ~ \lambda\in\dot\Sigma_1$,
    which, in terms of the matrix entries, reads as follows:
    \begin{subequations}\label{sym_s_bar_sigma1}
    \begin{alignat}{3} \label{sym_s_bar_sigma1-a}
    s_{11}(\lambda_+)&=\overline{s_{22}(\lambda_+)},\\\label{sym_s_bar_sigma1-b}
    s_{12}(\lambda_+)&=\overline{s_{21}(\lambda_+)}.
    \end{alignat}
    \end{subequations}
    
    \item $|s_{11}(\lambda_+)|^2-|s_{21}(\lambda_+)|^2=1$ for $\lambda\in \dot\Sigma_1$.

    \item $\big| \frac{s_{21}(\lambda_+)}{s_{11}(\lambda_+)} \big|\leq 1$ for $\lambda\in \dot\Sigma_1$.
  
    Notice that $\big| \frac{s_{21}(\lambda_+)}{s_{11}(\lambda_+)} \big|= 1$ for $\lambda\in \dot\Sigma_1$ iff $s_{11}(\lambda_+)=\infty$.

    \item 
    \begin{subequations}\label{sym_s_(minus)_bar_sigma_1)}
    \begin{alignat}{4}\label{sym_s_(minus)_bar_sigma_1)-a}
    s_{11}(\lambda_-)&=\overline{s_{22}(\lambda_-)},\qquad \lambda\in\dot\Sigma_1,\\\label{sym_s_(minus)_bar_sigma_1)-b}
    s_{12}(\lambda_-)&=\overline{s_{21}(\lambda_-)},\qquad \lambda\in\dot\Sigma_1.
    \end{alignat}
    \end{subequations}
    
    \item 
    \begin{equation}\label{sym_from_Psi_sigma_i}
    \Phi_j(\lambda_+)=\ii\Phi_j(\lambda_-) \sigma_1, \qquad \lambda\in\dot\Sigma_j.
    \end{equation}
    
    \item 
     \begin{subequations}\label{sym_from_Psi_sigma_i_}
     \begin{alignat}{4}\label{sym_from_Psi_sigma_i_a}
     \Phi_1^{(1)}(\lambda_+)&=\ii\Phi_1^{(2)}(\lambda_-), \qquad \lambda\in\dot\Sigma_1,\\\label{sym_from_Psi_sigma_i_b}
     \Phi_2^{(2)}(\lambda_+)&=\ii\Phi_2^{(1)}(\lambda_-), \qquad \lambda\in\dot\Sigma_2.
     \end{alignat}
    \end{subequations}
    
    \item  
    \begin{subequations}\label{sym_s11}
    \begin{alignat}{3} \label{sym_s11-a}
    s_{11}(\lambda_+)&=s_{22}(\lambda_-),\quad  & \lambda\in\dot\Sigma_1,\\\label{sym_s11-b}
    s_{11}(\lambda_+)&=-\ii s_{21}(\lambda_-),\quad  & \lambda\in\dot\Sigma_0,\\\label{sym_s11-c}
    s_{11}(\lambda_-)&=\ii s_{21}(\lambda_+),\quad  & \lambda\in\dot\Sigma_0.
    \end{alignat}
    \end{subequations}
    
    \item $\big| \frac{s_{21}(\lambda_+)}{s_{11}(\lambda_+)} \big|=1$ for $\lambda\in \dot\Sigma_0$.
    
    \item 
    \begin{equation}\label{sym_tildePhi_bar_sigma_i}
    \overline{\tilde\Phi_j(\lambda_+)}=\sigma_1\tilde\Phi_j(\lambda_+)\sigma_1, \qquad \lambda\in\dot\Sigma_j.
    \end{equation}

    \item 
    \begin{subequations}\label{lim_con}
    \begin{alignat}{3}
    \tilde\Phi_1^{(1)}(\lambda_-)=\sigma_1\tilde\Phi_1^{(2)}(\lambda_+),\quad \lambda\in\Sigma_1,\\
    \tilde\Phi_2^{(2)}(\lambda_-)=\sigma_1\tilde\Phi_2^{(1)}(\lambda_+),\quad \lambda\in\Sigma_2.
    \end{alignat}
    \end{subequations}
    
    \item 
    \begin{equation}\label{sym_bar}
    \overline{(D_j^{-1}\tilde\Phi_j)(\lambda_+)}=\ii(D_j^{-1}\tilde\Phi_j)(\lambda_+)\sigma_1, \quad \lambda\in\dot\Sigma_j.
    \end{equation}

    \item 
    \begin{subequations}\label{sym_D-tilde-phi} 
    \begin{alignat}{4}\label{sym_D-tilde-phi-a}
    D_j^{-1}(\lambda_-)\tilde\Phi_j^{(j)}(\lambda_-)&=(-\ii D_j^{-1}(\lambda_+)\tilde\Phi_j(\lambda_+)\sigma_1)^{(j)},\quad &\lambda\in\dot\Sigma_1,\\ \label{sym_D-tilde-phi-b}
    D_2^{-1}(\lambda_-)\tilde\Phi_2^{(2)}(\lambda_-)&=(-\ii D_2^{-1}(\lambda_+)\tilde\Phi_2(\lambda_+)\sigma_1)^{(2)},\quad &\lambda\in\dot\Sigma_0,\\ \label{sym_D-tilde-phi-c}
    D_1^{-1}(\lambda_-)\tilde\Phi_1^{(1)}(\lambda_-)&=D_1^{-1}(\lambda_+)\tilde\Phi_1^{(1)}(\lambda_+),\quad &\lambda\in\dot\Sigma_0.
    \end{alignat}
    \end{subequations}

    \item 
    \begin{subequations}\label{sym-D-tildePhi-comb}
    \begin{alignat}{4}
    (D_1^{-1}\tilde \Phi_1)((-\lambda)_+)&=\sigma_3\overline{(D_1^{-1}\tilde\Phi_1)(\lambda_+)},\quad &\lambda\in\dot\Sigma_1,\\
    (D_2^{-1}\tilde \Phi_2)((-\lambda)_+)&=-\sigma_3\overline{(D_2^{-1}\tilde\Phi_2)(\lambda_+)},\quad &\lambda\in\dot\Sigma_2.
    \end{alignat}
    \end{subequations}
    
    \item 
    \begin{equation}\label{sym_s11_Sigma_1_comb}
    s_{11}((-\lambda)_+)=\overline{s_{11}(\lambda_+)}, \quad\lambda\in\dot\Sigma_1.
    \end{equation}
    \end{enumerate}
\end{corollary} 

\begin{proof}
\begin{enumerate}
    \item Substitute \eqref{sym_Phi_bar} into \eqref{scat}.
    
    \item This follows from the fact that $\det{s(\lambda_\pm)=1}$ for all $\lambda\in\Sigma_1$ and \eqref{sym_s_bar_sigma1}.
    
    \item Dividing the previous equality by $|s_{11}(\lambda_+)|^2$, we obtain $1-\big|\frac{s_{21}(\lambda_+)}{s_{11}(\lambda_+)}\big|^2=\big|\frac{1}{s_{11}(\lambda_+)}\big|^2\geq0$. Hence $\big|\frac{s_{21}(\lambda_+)}{s_{11}(\lambda_+)}\big|\leq 1$.
    
    \item Combine \eqref{sym_s_bar_sigma1} and \eqref{sym_s_barbar}.
    
    \item Combine \eqref{sym_Phi_bar} and \eqref{sym_Phi_barbar_Sigma_i}.
    
    \item Rewrite \eqref{sym_from_Psi_sigma_i} columnwise.
    
    \item Substituting \eqref{sym_from_Psi_sigma_i} into \eqref{scatcoeff-a} leads to \eqref{sym_s11}. Notice that in proving \eqref{sym_s11-b} and  \eqref{sym_s11-c}
    we use the fact that $\Phi_1^{(1)}$ is analytic on $\Sigma_0$.
    
    \item Using the previous result for the first equality and \eqref{sym_s11_barbar} for the second one, we get $\big| \frac{s_{21}(\lambda_+)}{s_{11}(\lambda_+)} \big|=\big| \frac{-\ii s_{11}(\lambda_-)}{s_{11}(\lambda_+)} \big|=\big| \frac{\overline{ s_{11}(\lambda_+)}}{s_{11}(\lambda_+)} \big|=1$.
    
    \item Observe that $\sqrt{-\frac{1}{\ii A_jk_j(\lambda_+)}-1}\cdot\frac{\lambda_+A_j}{1+\ii A_jk_j(\lambda_+)}=-\ii\sqrt{\frac{1}{\ii A_jk_j(\lambda_+)}-1}$ and $\sqrt{\frac{1}{\ii A_jk_j(\lambda_+)}-1}\cdot\frac{\lambda_+A_j}{1-\ii A_jk_j(\lambda_+)}=\ii\sqrt{-\frac{1}{\ii A_jk_j(\lambda_+)}-1}$ imply $\overline{D_j^{-1}(\lambda_+)}=\ii D_j^{-1}(\lambda_+)\sigma_1$ and $\overline{D_j(\lambda_+)}=-\ii \sigma_1D_j(\lambda_+)$, and \eqref{sym_k_i-d} imply $\overline{Q_j(\lambda_+)}=-Q_j(\lambda_+)$ for $\lambda\in\dot\Sigma_j$. Combining this with \eqref{sym_Phi_bar} and using the connection between $\Phi_j$ and $\tilde\Phi_j$, we obtain \eqref{sym_tildePhi_bar_sigma_i}.
    
    \item Combine \eqref{sym_tildePhi_bar_sigma_i} and \eqref{sym_tildePhi_barbar_Sigma_i}.
    
         \item Combine $\overline{D_j^{-1}(\lambda_+)}=\ii D_j^{-1}(\lambda_+)\sigma_1$ and  \eqref{sym_tildePhi_bar_sigma_i}.
    
    \item
    Use \eqref{sym_bar} combined with \eqref{sym_bar_Sigma_i}  for the first two equalities and the fact that $k_1(\lambda)$ is analytic on $\dot\Sigma_0$ for the last one.
    \item Combine \eqref{sym_bar} and \eqref{sym_(minus)}.
    
    \item \eqref{sym_s_(minus)_sigma1} implies $s_{22}(\lambda_+)=s_{11}((-\lambda)_+)$. Combine this with \eqref{sym_s_bar_sigma1-a}.
\end{enumerate}
\end{proof}

\subsection{Limits of the eigenfunctions and scattering coefficients from below and above the branch cut}  \label{sec:limits}
Recall that $k_j(\lambda)$ is analytic in $\mathbb{C}\setminus\Sigma_j$ and discontinuous across $\Sigma_j$.

\begin{notations*} It will be useful in what follows to introduce the following notations (for $\lambda \in\Sigma_j$):

\[k_j^+(\lambda):=k_j(\lambda_+)=\lim_{\epsilon\downarrow 0}k_j(\lambda+\ii\epsilon), \qquad k_j^-(\lambda):=k_j(\lambda_-)=\lim_{\epsilon\downarrow 0}k_j(\lambda-\ii\epsilon).\]
Similarly, 
\[\tilde\Phi_1^{(1)+}(\lambda):=\tilde\Phi_1^{(1)}(\lambda_+)=\lim_{\epsilon\downarrow 0}\tilde\Phi_1^{(1)}(\lambda+\ii\epsilon),\qquad
\tilde\Phi_1^{(1)-}(\lambda):=\tilde\Phi_1^{(1)}(\lambda_-)=\lim_{\epsilon\downarrow 0}\tilde\Phi_1^{(1)}(\lambda-\ii\epsilon).\]
\end{notations*}
Observe that 
\begin{subequations}\label{lim_k}
\begin{alignat}{3}
k_j^-(\lambda)&=-k_j^+(\lambda),\quad \lambda\in\Sigma_1,\\
k_1^-(\lambda)&=k_1^+(\lambda)=k_1(\lambda),\quad \lambda\in\Sigma_0,\\
k_2^-(\lambda)&=-k_2^+(\lambda),\quad \lambda\in\Sigma_0.
\end{alignat}
\end{subequations}
Combining \eqref{sym_tildePhi_bar_sigma_i} and \eqref{sym_tildePhi_barbar_Sigma_i} we have
\begin{subequations}\label{lim_con_}
\begin{alignat}{3}
\tilde\Phi_1^{(1)-}(\lambda)=\sigma_1\tilde\Phi_1^{(2)+}(\lambda),\quad \lambda\in\Sigma_1,\\
\tilde\Phi_2^{(2)-}(\lambda)=\sigma_1\tilde\Phi_2^{(1)+}(\lambda),\quad \lambda\in\Sigma_2.
\end{alignat}
\end{subequations}


\subsection{Discrete spectrum and zeros of scattering coefficients}  \label{sec:discrete spectrum}
Multiplying  \eqref{Lax-x} by $\begin{pmatrix}
0&-1\\1&0
\end{pmatrix}$ we arrive at the spectral problem for a weighted Dirac operator:
\begin{equation}\label{Dirac}
    \frac{2}{m}\left( \begin{pmatrix}
0&-1\\1&0
\end{pmatrix} \Phi_x+\frac{1}{2}\begin{pmatrix}
0&1\\1&0
\end{pmatrix}\Phi\right)=\lambda \Phi, \quad x\in(-\infty,\infty). 
\end{equation}
Since $\lim_{x\to(-1)^j\infty}m(x,t)=A_j\ne 0$, this operator can be viewed
as a self-adjoint operator in $L^2(-\infty,\infty)$ and thus its spectrum in real.

Observe that for $\lambda\in\dot\Sigma_1$, both $k_j(\lambda)$, $j=1,2$ are real-valued and hence the eigenfunctions $\Phi_j$ are bounded but not square integrable near $(-1)^j\infty$. Since they are related by a matrix independent on $x$ and $t$, $\Phi_j$ are  bounded and not square integrable near $\pm\infty$. Hence $\dot\Sigma_1$ comprise the continuous spectrum.

For $\lambda\in (-1/A_2, 1/A_2)$, $\Phi_1^{(1)}$ decays (exponentially fast) as
$x\to -\infty$ and $\Phi_2^{(2)}$ decays (exponentially fast) as
$x\to +\infty$; hence the 
the eigenvalues in $(-1/A_2, 1/A_2)$
coincides with the zeros of $s_{11}(\lambda)=\det(\Phi_1^{(1)},\Phi_2^{(2)})$.

Note that since $|s_{11}(\lambda_+)|^2-|s_{21}(\lambda_+)|^2=1$ for $\lambda\in \dot\Sigma_1$ (see Corollary \ref{cor: sym}), we have $s_{11}(\lambda_+)\neq 0$ for $\lambda\in \dot\Sigma_1$. 

Let's show that $s_{11}(\lambda_+)\neq 0$ as well as 
$s_{21}(\lambda_+)\neq 0$ for $\lambda\in \dot\Sigma_0$
(the similar result for $\lambda_-$ will then follow   from the symmetry \eqref{sym_s_barbar}).
Indeed, we have  $|\frac{s_{21}}{s_{11}}(\lambda_\pm)|=1$ for $\lambda\in\dot\Sigma_0$ (see Corollary \ref{cor: sym}). 
Hence  $s_{11}(\lambda_{0+})s_{21}(\lambda_{0+})=0$ iff
$s_{11}(\lambda_{0+})=0$ and $s_{21}(\lambda_{0+})=0$ simultaneously.
But $s_{11}(\lambda_{0+})=0$ implies that $\Phi_1^{(1)}(\lambda_{0+})$ and $\Phi_2^{(2)}(\lambda_{0+})$ are dependent. Silarly, $s_{21}(\lambda_{0+})=0$ implies that $\Phi_1^{(1)}(\lambda_{0+})$ and $\Phi_2^{(1)}(\lambda_{0+})$ are dependent. Hence $\Phi_2^{(1)}(\lambda_{0+})$ and $\Phi_2^{(2)}(\lambda_{0+})$ are dependent, which 
contradicts the fact that  $\det\Phi_{0,2}\equiv 1$ (the latter follows from evaluating
$\det\Phi_{0,2}(x,t,\lambda)$ as $x\to\infty$ and using the fact that 
the determinant of a matrix composed by two vector solutions of \eqref{Dirac}
does not depend on $x$).

\textit{Assumption. }We will assume that $s_{11}(\lambda)$ has a finite number of zeros on $\mathbb{R}\setminus\Sigma_2$. Since $s_{11}$ is analytic on $\mathbb{C}\setminus\Sigma_2$, the uniqueness theorem implies that the sufficient condition is $s_{11}(\pm\frac{1}{A_2})\neq0$.

Let $\{\lambda_k\}_{k=1}^n$ be the zeros of $s_{11}(\lambda)$.
For such $\lambda_k$ we have \[\Phi_1^{(1)}(\lambda_k)=b_k\Phi_2^{(2)}(\lambda_k), \quad b_k\coloneqq b(\lambda_k).\]
\begin{proposition}\label{prop:simple-zeros}
The zeros of $s_{11}(\lambda)$ are simple.
\end{proposition}

\begin{proof}
We will denote by $'$ the dirivative w.r.t. $\lambda$.

Using the definition of $s_{11}(\lambda)$ we have
\[s_{11}'(\lambda)=\det(\Phi_1^{(1)},\Phi_2^{(2)})'(\lambda)=\det((\Phi')_1^{(1)},\Phi_2^{(2)})(\lambda)+\det(\Phi_1^{(1)},(\Phi')_2^{(2)})(\lambda).\]
Since $\Phi_j^{(j)}$ solves \eqref{Lax-x}, we have 
\[(\Phi')_{jx}^{(j)}=U(\Phi')_j^{(j)}+m\begin{pmatrix}
0&1\\-1&0
\end{pmatrix}\Phi_j^{(j)},\]
and, using the fact that $\det(U(\Phi')_1^{(1)},\Phi_2^{(2)})=-\det((\Phi')_1^{(1)},U\Phi_2^{(2)})$, we have
\[\frac{d}{dx}\det((\Phi')_1^{(1)},\Phi_2^{(2)})=\det\left(\begin{pmatrix}
0&m\\-m&0
\end{pmatrix}\Phi_1^{(1)},\Phi_2^{(2)}\right),\]
and
\[\frac{d}{dx}\det(\Phi_1^{(1)},(\Phi')_2^{(2)})=-\det\left(\begin{pmatrix}
0&m\\-m&0
\end{pmatrix}\Phi_2^{(2)},\Phi_1^{(1)}\right).\]
Evaluating at $\lambda=\lambda_k$ and using $\Phi_1^{(1)}(\lambda_k)=b_k\Phi_2^{(2)}(\lambda_k)$, we get
\[\frac{d}{dx}\det((\Phi')_1^{(1)},\Phi_2^{(2)})(\lambda_k)=b_km\det\left(\begin{pmatrix}
0&1\\-1&0
\end{pmatrix}\Phi_2^{(2)}(\lambda_k),\Phi_2^{(2)}(\lambda_k)\right),\]

\[\frac{d}{dx}\det(\Phi_1^{(1)},(\Phi')_2^{(2)})(\lambda_k)=-b_km\det\left(\begin{pmatrix}
0&1\\-1&0
\end{pmatrix}\Phi_2^{(2)}(\lambda_k),\Phi_2^{(2)}(\lambda_k)\right).\]
Using the symmetry \eqref{sym_phi_barbar-notin_Sigma_i} and observing that $\lambda_k\in\mathbb{R}$, we have
\[\det(\begin{pmatrix}
0&1\\-1&0
\end{pmatrix}\Phi_2^{(2)}(\lambda_k),\Phi_2^{(2)}(\lambda_k))=-(|(\Phi_2)_{22}|^2+|(\Phi_2)_{12}|^2)(\lambda_k)\]
and hence
\[\frac{d}{dx}\det((\Phi')_1^{(1)},\Phi_2^{(2)})(\lambda_k)=b_k\int_x^{\infty}m(|(\Phi_2)_{22}|^2+|(\Phi_2)_{12}|^2)d\tau,\]
\[\frac{d}{dx}\det(\Phi_1^{(1)},(\Phi')_2^{(2)})(\lambda_k)=b_k\int^x_{-\infty}m(|(\Phi_2)_{22}|^2+|(\Phi_2)_{12}|^2)dx.\]
It follows that 
\[s_{11}'(\lambda_k)=b_k\int^\infty_{-\infty}m(|(\Phi_2)_{22}|^2+|(\Phi_2)_{12}|^2)dx,\]
and thus $s'_{11}(\lambda_k)\neq0$.

\end{proof}

Observe that due to the symmetry \eqref{sym_s11_minus_notSigma_2}, if  $s_{11}(\lambda_k)=0$, then $s_{11}(-\lambda_k)=0$ as well. Since, according to  Proposition \ref{prop:simple-zeros},
all zeros of $s_{11}$ are simple, it follows that  $s_{11}(0)\ne 0$. 
This fact will also be discussed in Subsection \ref{Eigen_0}.

\subsection{Behaviour at the branch points}  \label{sec:branch points}

Observe that $k_j(\pm\frac{1}{A_j})=0$.
\begin{proposition}\label{prop:sing_beh} $\tilde\Phi_j(x,t,\lambda)$ has the following behaviour at the branch points
\[\tilde\Phi_j(x,t,\lambda)=\frac{\ii\alpha_j(x,t)}{\omega_j^+(\lambda)}\begin{pmatrix}
1&1\\-1&-1
\end{pmatrix}+\begin{pmatrix}
a_j(x,t)&b_j(x,t)\\
b_j(x,t)&a_j(x,t)
\end{pmatrix}+\ord(\sqrt{\lambda-\frac{1}{A_j}}),\quad \lambda\to\frac{1}{A_j},\]
\[\tilde\Phi_j(x,t,\lambda)=\frac{\alpha_j(x,t)}{\omega_j^-(\lambda)}\begin{pmatrix}
1&-1\\1&-1
\end{pmatrix}+\begin{pmatrix}
a_j(x,t)&-b_j(x,t)\\
-b_j(x,t)&a_j(x,t)
\end{pmatrix}+\ord(\sqrt{\lambda+\frac{1}{A_j}}),\quad \lambda\to-\frac{1}{A_j},\]
whith some real-valued  $\alpha_j
(x,t)$, $a_j(x,t)$, and $b_j(x,t)$, $j=1,2$.
\end{proposition}

\begin{proof}
Recall that $\omega_j^+(\lambda)=\sqrt{\lambda-\frac{1}{A_j}}$ with a branch cut on $[\frac{1}{A_j},\infty)$ and $\omega_j^+(0)=\frac{\ii}{\sqrt{A_j}}$, and $\omega_j^-(\lambda)=\sqrt{\lambda+\frac{1}{A_j}}$ with a branch cut on $(-\infty,-\frac{1}{A_j}]$ and $\omega_j^-(0)=\frac{1}{\sqrt{A_j}}$.

First, consider the behavior of the eigenfunctions near $\frac{1}{A_j}$.
Introduce $\tilde{\tilde\Phi}_j(x,t,\lambda)$ such that $\tilde\Phi_j(x,t,\lambda)=W^+\tilde{\tilde\Phi}_j(x,t,\lambda)$ with $W^+=\begin{pmatrix}
1&\frac{\ii}{\omega_j^+(\lambda)}\\
1&-\frac{\ii}{\omega_j^+(\lambda)}
\end{pmatrix}$. Then  $\tilde{\tilde\Phi}_j(x,t,\lambda)$
solves the following integral equation:
\[\tilde{\tilde\Phi}_j(x,t,\lambda)=\frac{1}{2}\begin{pmatrix}
1&1\\
-\ii\omega_j^+(\lambda)&\ii\omega_j^+(\lambda)
\end{pmatrix} + \int_{(-1)^i\infty}^x A^{-1} \eul^{\frac{\ii}{2}k_j(\lambda)\int_x^\xi m d\tau\sigma_3}\hat U_j A \tilde{\tilde\Phi}_j \eul^{-\frac{\ii}{2}k_j(\lambda)\int_x^\xi m d\tau\sigma_3}.\]
The kernel of this equation and hence $\tilde{\tilde\Phi}_j$ has no singularity at $\frac{1}{A_j}$.
Hence 
\[\tilde\Phi_j(x,t,\lambda)=\frac{\ii}{\omega_j^+(\lambda)}\begin{pmatrix}
\tilde{\tilde c}_j&\tilde{\tilde d}_j\\-\tilde{\tilde c}_j&-\tilde{\tilde d}_j
\end{pmatrix}+\begin{pmatrix}
a_j&b_j\\c_j&d_j
\end{pmatrix}+\ord\left(\sqrt{\lambda-\frac{1}{A_j}}\right),\quad \lambda\to\frac{1}{A_j}.\]
Using \eqref{sym_tildePhi_barbar_Sigma_i}, we get $\tilde{\tilde c}_j$, $\tilde{\tilde d}_j\in\mathbb{R}$ and $a_j,~b_j,~c_j,~d_j\in\mathbb{R}$. Then,
using \eqref{sym_tildePhi_bar_sigma_i}, we get $\tilde{\tilde c}_j=\tilde{\tilde d}_j$ and $a_j=d_j,~c_j=b_j$; thus
\[\tilde\Phi_j(x,t,\lambda)=\frac{\ii\alpha_j(x,t)}{\omega_j^+(\lambda)}\begin{pmatrix}
1&1\\-1&-1
\end{pmatrix}+\begin{pmatrix}
a_j&b_j\\b_j&a_j
\end{pmatrix}+\ord\left(\sqrt{\lambda-\frac{1}{A_j}}\right),\quad \lambda\to\frac{1}{A_j}.\]

In order to get the simiular result for $-\frac{1}{A_j}$, we use $W^-=\begin{pmatrix}
\frac{\ii}{\omega_j^-(\lambda)}&1\\
\frac{\ii}{\omega_j^-(\lambda)}&-1
\end{pmatrix}$ instead of $W^+$, which leads to 
\[\tilde\Phi_j(x,t,\lambda)=\frac{\beta_j(x,t)}{\omega_j^-(\lambda)}\begin{pmatrix}
-1&1\\-1&1
\end{pmatrix}+\begin{pmatrix}
\hat a_j&\hat b_j\\\hat b_j&\hat a_j
\end{pmatrix}+\ord(\sqrt{\lambda+\frac{1}{A_j}}),\quad \lambda\to-\frac{1}{A_j}.\]

Finally, using \eqref{sym_tildePhi_minus_sigma_i} and \eqref{sym_tildePhi_(minus)_sigma_i}, we get $\alpha_j=-\beta_j$ and $a_j=\hat a_j$ and $b_j=-\hat b_j$.
\end{proof}

Evaluating  $D_j^{-1}(\lambda)$ near $\pm \frac{1}{A_j}$ gives
\begin{proposition}\label{prop:sing_beh_D}
$D_j^{-1}(\lambda)$ has the following behaviour at the branch points:
\[D_j^{-1}(\lambda)=\frac{\eul^{\frac{3\pi\ii}{4}}}{(2A_j)^{\frac{1}{4}}\nu_j^+(\lambda)}\begin{pmatrix}
1&1\\1&1
\end{pmatrix}+\frac{\ii \eul^{\frac{3\pi\ii}{4}}(2A_j)^{\frac{1}{4}}\nu_j^+(\lambda)}{2}\begin{pmatrix}
1&-1\\-1&1
\end{pmatrix}+\ord((\lambda-\frac{1}{A_j})^{\frac{3}{4}}),\quad \lambda\to\frac{1}{A_j}\]
and 
\[D_j^{-1}(\lambda)=\frac{\ii}{(2A_j)^{\frac{1}{4}}\nu_j^-(\lambda)}\begin{pmatrix}
-1&1\\1&-1
\end{pmatrix}+\frac{\ii (2A_j)^{\frac{1}{4}}\nu_j^-(\lambda)}{2}\begin{pmatrix}
1&1\\1&1
\end{pmatrix}+\ord((\lambda+\frac{1}{A_j})^{\frac{3}{4}}),\quad \lambda\to-\frac{1}{A_j}.\]
Here $\nu_j^+(\lambda)=(\lambda-\frac{1}{A_j})^{\frac{1}{4}}$ with the branch cut  $(\frac{1}{A_j},\infty)$ and $\nu_j^+(0)=\frac{\eul^{\frac{\pi\ii}{4}}}{(A_j)^{\frac{1}{4}}}$,
and $\nu_j^-(\lambda)=(\lambda+\frac{1}{A_j})^{\frac{1}{4}}$ with the branch cut  $(-\infty,-\frac{1}{A_j})$ and $\nu_j^-(0)=\frac{1}{(A_j)^{\frac{1}{4}}}$
(observe that $(\nu_j^\pm(\lambda))^2=\omega_j^\pm(\lambda)$).
\end{proposition}


\section{Riemann--Hilbert problems}\label{sec:3}
\subsection{RH problem parametrized by $\BS{(x,t)}$}

\begin{notations*}
We  denote
\begin{equation}\label{rho}
    \rho(\lambda):=\frac{s_{21}(\lambda_+)}{s_{11}(\lambda_+)},\quad\lambda\in\dot\Sigma_1\cup\dot\Sigma_0.
\end{equation}

Observe that Corollary \ref{cor: sym} implies that 
\begin{subequations}\label{prop-rho}
\begin{alignat}{4}\label{prop-rho-a}
|\rho(\lambda)|\leq 1,\quad\lambda\in\dot\Sigma_1,\\\label{prop-rho-b}
|\rho(\lambda)|= 1,\quad\lambda\in\dot\Sigma_0.
\end{alignat}

\end{subequations}
\end{notations*}

Motivated by the analytic properties of eigenfunctions and scattering coefficients, we introduce the matrix-values function

\begin{subequations}\label{M}
\begin{equation}\label{M-a}
M(x,t,\lambda)= \left( \frac{(D_1^{-1}\tilde\Phi_1^{(1)})(x,t,\lambda)}{s_{11}(\lambda)\eul^{p_1(x,t,\lambda)-p_2(x,t,\lambda)}},(D_2^{-1}\tilde\Phi_2^{(2)})(x,t,\lambda)\right),\qquad \lambda\in \mathbb{C}\setminus\Sigma_2,
\end{equation}
meromorphic in $\mathbb{C}\setminus\Sigma_2$,
where $p_j$, $j=1,2$ are defined in \eqref{p_i}.

Observe that $D_j^{-1}(\lambda)\tilde\Phi_j(x,t,\lambda)=\Phi_j(x,t,\lambda)\eul^{Q_j(x,t,\lambda)}$ and thus $M(x,t,\lambda)$ can be written as 
\begin{equation}\label{M-b}
M(x,t,\lambda)= \left( \frac{\Phi_1^{(1)}(x,t,\lambda)}{s_{11}(\lambda)},\Phi_2^{(2)}(x,t,\lambda)\right)\eul^{p_2(x,t,\lambda)\sigma_3}.
\end{equation}
\end{subequations}
It follows that $\det M\equiv 1$.
\subsubsection{Jump matrix}

Since $(D_1^{-1}\tilde\Phi_1^{(1)})(\lambda)$ is analytic in $\mathbb{C}\setminus\Sigma_1$, the limiting values $M^\pm$ of $M$ as $\lambda$ approaches $\Sigma_2$ 
from ${\mathbb C}^\pm$ can be expressed as follows:

\[
M^\pm(x,t,\lambda)\coloneqq M(x,t,\lambda_\pm)= \left( \frac{(D_1^{-1}\tilde\Phi_1^{(1)})(x,t,\lambda_\pm)}{s_{11}(\lambda_\pm)\eul^{p_1(x,t,\lambda_\pm)-p_2(x,t,\lambda_\pm)}},(D_2^{-1}\tilde\Phi_2^{(2)})(x,t,\lambda_\pm)\right),\quad \lambda\in\dot\Sigma_1,
\]

\[
M^\pm(x,t,\lambda)\coloneqq M(x,t,\lambda_\pm)= \left( \frac{(D_1^{-1}\tilde\Phi_1^{(1)})(x,t,\lambda)}{s_{11}(\lambda_\pm)\eul^{p_1(x,t,\lambda)-p_2(x,t,\lambda_\pm)}},(D_2^{-1}\tilde\Phi_2^{(2)})(x,t,\lambda_\pm)\right),\quad \lambda\in\dot\Sigma_0.
\]

\begin{proposition}
$M^+$ and $M^-$ are related as follows:
\[M^+(x,t,\lambda) =M^-(x,t,\lambda) J(x,t,\lambda), \quad \lambda\in\dot\Sigma_1\cup\dot\Sigma_0,\]
where
\begin{subequations}\label{jump_}
\begin{alignat}{4}\label{jump_J}
J(x,t,\lambda)=\begin{pmatrix}
0&\ii\\
\ii&0
\end{pmatrix}\begin{pmatrix}
\eul^{-p_2(x,t,\lambda_+)}&0\\
0&\eul^{p_2(x,t,\lambda_+)}
\end{pmatrix}J_0(\lambda)
\begin{pmatrix}
\eul^{p_2(x,t,\lambda_+)}&0\\
0&\eul^{-p_2(x,t,\lambda_+)}
\end{pmatrix}
\end{alignat}
with
\begin{equation}\label{jump_J0}
J_0(\lambda)=\begin{cases} \begin{pmatrix}
1-|\rho(\lambda)|^2&-\overline{\rho(\lambda)}\\
\rho(\lambda)&1
\end{pmatrix}
, \quad \lambda\in\dot\Sigma_1,\\
\begin{pmatrix}
0&-\frac{1}{\rho(\lambda)}\\
\rho(\lambda)&1
\end{pmatrix}, \quad \lambda\in\dot\Sigma_0.
\end{cases}
\end{equation}
\end{subequations}
\end{proposition}

\begin{proof}
\textit{(i) $\lambda\in\dot\Sigma_1$.}
Considering \eqref{scat_2} columnwise, rearranging the columns and using \eqref{sym_D-tilde-phi-a} for $\lambda\in\dot\Sigma_1$, we obtain
\begin{equation}\label{jump-sigma-1}
M^+(x,t,\lambda) =M^-(x,t,\lambda)\ii\begin{pmatrix}
\frac{\tilde s_{21}(x,t,\lambda_+)\tilde s_{11}(x,t,\lambda_-)}{\tilde s_{11}(x,t,\lambda_+)\tilde s_{22}(x,t,\lambda_+)}&\frac{\tilde s_{11}(x,t,\lambda_-)}{\tilde s_{22}(x,t,\lambda_+)}\\
1-\frac{\tilde s_{21}(x,t,\lambda_+)\tilde s_{12}(x,t,\lambda_+)}{\tilde s_{11}(x,t,\lambda_+)\tilde s_{22}(x,t,\lambda_+)}&-\frac{\tilde s_{12}(x,t,\lambda_+)}{\tilde s_{22}(x,t,\lambda_+)}
\end{pmatrix}.
\end{equation}
Since $\eul^{p_1(x,t,\lambda_-)-p_2(x,t,\lambda_-)}=\eul^{p_2(x,t,\lambda_+)-p_1(x,t,\lambda_+)}$, from \eqref{sym_s_barbar} and \eqref{sym_s_bar_sigma1-a} we have $\frac{\tilde s_{11}(\lambda_-)}{\tilde s_{22}(\lambda_+)}=\frac{ s_{11}(\lambda_-)}{ s_{22}(\lambda_+)}=1$.
Moreover, using the definition \eqref{rho} of $\rho(\lambda)$ and \eqref{sym_s_bar_sigma1}, we have $\overline{\rho(\lambda)}=\frac{s_{12}(\lambda_+)}{s_{22}(\lambda_+)}$. Hence we can rewrite the jump condition \eqref{jump-sigma-1} 
as \eqref{jump_J}  with \eqref{jump_J0}.

\textit{(ii) $\lambda\in\dot\Sigma_0$.}
Considering  \eqref{scat-col_} columnwise, rearranging the columns and using \eqref{sym_D-tilde-phi-b} and \eqref{sym_D-tilde-phi-c} for $\lambda_+\in\dot\Sigma_0$, we obtain
\begin{equation}\label{jump-sigma-0}
M^+(x,t,\lambda) =M^-(x,t,\lambda)\ii\begin{pmatrix}
\frac{\tilde s_{21}(x,t,\lambda_+)}{\tilde s_{11}(x,t,\lambda_+)}&1\\
0&-\frac{\tilde s_{11}(x,t,\lambda_+)}{\tilde s_{21}(x,t,\lambda_+)}
\end{pmatrix}.
\end{equation}

Then, using the definition of $\rho(\lambda)$ together with  \eqref{sym_s11-c} and \eqref{sym_s11-b}, we can rewrite the jump condition \eqref{jump-sigma-0} as \eqref{jump_J} with \eqref{jump_J0}.

\end{proof}

\begin{remark}
Notice that
      \begin{equation}\label{det_J}
    \det J \equiv 1
    \end{equation}
    and that $J_0(\lambda)$ (and hence $J$) is continuous at $\pm\frac{1}{A_1}$ if $|\rho(\pm\frac{1}{A_1})|=1$ and $\rho(\pm\frac{1}{A_1}+0)=\rho(\pm\frac{1}{A_1}-0)$, and discontinuous otherwise.
\end{remark}


\subsubsection{Normalization condition at $\lambda\to\infty$.}
\begin{proposition}
As $\lambda\to\infty$:
\begin{equation}\label{M-norm}
M(x,t,\lambda)=\begin{cases}\sqrt{\frac{1}{2}}
\begin{pmatrix}
-1&\ii\\
\ii&-1
\end{pmatrix}+ \ord(\frac{1}{\lambda}), \quad \lambda\to\infty,~ \lambda\in\mathbb{C}^+,\\
\sqrt{\frac{1}{2}}
\begin{pmatrix}
1&\ii\\
\ii&1
\end{pmatrix}+ \ord(\frac{1}{\lambda}), \quad \lambda\to\infty,~ \lambda\in\mathbb{C}^-.
\end{cases}
\end{equation}
\end{proposition}

\begin{proof}
 Expanding $D_j^{-1}(\lambda)$ \eqref{D-j} as $\lambda\to\infty$, we get  \[D_j^{-1}(\lambda)=\begin{cases}\sqrt{\frac{1}{2}}
\begin{pmatrix}
-1&\ii\\
\ii&-1
\end{pmatrix}+\ord(\frac{1}{\lambda}), \quad \lambda\to\infty,~ \lambda\in\mathbb{C}^+,\\
\sqrt{\frac{1}{2}}
\begin{pmatrix}
1&\ii\\
\ii&1
\end{pmatrix}+\ord(\frac{1}{\lambda}), \quad \lambda\to\infty,~ \lambda\in\mathbb{C}^-.
\end{cases}\]
Recalling that $\left(\begin{smallmatrix}
\widetilde\Phi_1^{(1)} &
\widetilde\Phi_2^{(2)}\end{smallmatrix}\right)\to I$ as $\lambda\to\infty$,
we have, for $\lambda\in\mathbb{C}^+$, 
\[(D_1^{-1}\tilde\Phi_1^{(1)})(\lambda)=\sqrt{\frac{1}{2}}\begin{pmatrix}
-1\\\ii
\end{pmatrix}+\ord(\frac{1}{\lambda}), \quad \lambda\to\infty,\]
\[(D_2^{-1}\tilde\Phi_2^{(2)})(\lambda)=\sqrt{\frac{1}{2}}\begin{pmatrix}
\ii\\-1
\end{pmatrix}+\ord(\frac{1}{\lambda}), \quad \lambda\to\infty.\]
Substituting this into \eqref{til_scatcoeff-a}, we get $\tilde s_{11}(\lambda)=1+\ord(\frac{1}{\lambda}),~\lambda\to\infty.$ 

Similarly, for  $\lambda\in\mathbb{C}^-$ we have
\[(D_1^{-1}\tilde\Phi_1^{(1)})(\lambda)=\sqrt{\frac{1}{2}}\begin{pmatrix}
1\\\ii
\end{pmatrix}+\ord(\frac{1}{\lambda}), \quad \lambda\to\infty,\]
\[(D_2^{-1}\tilde\Phi_2^{(2)})(\lambda)=\sqrt{\frac{1}{2}}\begin{pmatrix}
\ii\\1
\end{pmatrix}+\ord(\frac{1}{\lambda}), \quad \lambda\to\infty,\]
and $\tilde s_{11}(\lambda)=1+\ord(\frac{1}{\lambda}),~\lambda\to\infty$.
Then the claim follows.
\end{proof}

\begin{remark}\label{rem_3.5.}
  In order to have a standard normalisation as $\lambda\to\infty$, we can introduce 
  \begin{equation}\label{tilde-M}
  \tilde M(x,t,\lambda):=\begin{cases}\sqrt{\frac{1}{2}}
    \begin{pmatrix}
    -1&-\ii\\
    -\ii&-1
    \end{pmatrix}M(x,t,\lambda), \quad \lambda\in\mathbb{C}^+,\\
    \sqrt{\frac{1}{2}}
    \begin{pmatrix}
    -1&-\ii\\
    -\ii&-1
    \end{pmatrix}M(x,t,\lambda)\ii\sigma_1, \quad \lambda\in\mathbb{C}^-.
    \end{cases}
    \end{equation}
    Then we have $\tilde M\to I$ at $\lambda\to\infty$.
    On the other hand, $\tilde M$ acquires an additional  jump 
    across $\lambda\in\mathbb{R}\setminus\Sigma_2$:
        \[
        \tilde M^+(x,t,\lambda) =\tilde M^-(x,t,\lambda) \tilde J(x,t,\lambda),~ \lambda\in\mathbb{R}\setminus\big\{\cup_{j=1,2} \{A_j^{-1}\}\cup \{-A_j^{-1}\}\big\}
        \]
    with
    \begin{equation*}
    \tilde J(x,t,\lambda)=\begin{cases}
    \tilde J_{\Sigma_j}(x,t,\lambda), \quad \lambda\in\dot\Sigma_j,\quad j=0,1\\
        \tilde J_{\mathbb{R}\setminus\Sigma_2}(x,t,\lambda), \quad \lambda\in\mathbb{R}\setminus\Sigma_2,
    \end{cases}
    \end{equation*}
    where $\tilde J_{\Sigma_j}(x,t,\lambda)=\eul^{-p_2(x,t,\lambda_+)\sigma_3}J_0(\lambda)\eul^{p_2(x,t,\lambda_+)\sigma_3}$, $j=0,1$
    and 
    $\tilde J_{\mathbb{R}\setminus\Sigma_2}(x,t,\lambda)=-\ii\sigma_1$.
        \end{remark} 
        \begin{remark}
    Using \eqref{til_scatcoeff-c}, we obtain $\tilde s_{21}(\lambda)=\ord(\frac{1}{\lambda})$ as $\lambda\to\infty$. Notice that $\rho(\lambda)=\frac{s_{21}(\lambda_+)}{s_{11}(\lambda_+)}=\frac{\tilde s_{21}(\lambda_+)}{\tilde  s_{11}(\lambda_+)}\eul^{-2p_2(x,t,\lambda_+)}$; since $p_2(x,t,\lambda_+)$ is purely imaginary for $\lambda\in\Sigma_2$, $\eul^{-2p_2(x,t,\lambda_+)}$ is bounded and thus $\rho(\lambda)=\ord(\frac{1}{\lambda})$ as $\lambda\to\infty$.
Consequently, 
    \[
    J_{0}(\lambda)=\begin{pmatrix}
    1&0\\0&1
    \end{pmatrix}+\ord(\frac{1}{\lambda}),\quad \lambda\to\pm\infty
    \]
and
    \[
    J(x,t,\lambda)=\begin{pmatrix}
    0&\ii\\\ii&0
    \end{pmatrix}+\ord(\frac{1}{\lambda}),\quad \lambda\to\pm\infty.
    \]
\end{remark}    

\subsubsection{Symmetries}
From the symmetry properties of the eigenfunctions and scattering functions
\eqref{sym_minus_}, \eqref{sym_barbar_notin_Sigma_i}, \eqref{sym_minus}, and \eqref{sym_bar_Sigma_i} it follows that 
\begin{subequations}\label{sym-M}
\begin{alignat}{4}
M(-\lambda)&=-\sigma_3 M(\lambda)\sigma_3,\qquad \overline{M(\overline{\lambda})}=-M(\lambda),\qquad \lambda\in\mathbb{C}\setminus\Sigma_2,\\
M((-\lambda)_-)&=-\sigma_3 M(\lambda_+)\sigma_3,\qquad \overline{M(\lambda_-)}=-M(\lambda_+),\qquad \lambda\in\dot\Sigma_1.
\end{alignat}
\end{subequations}
where $M(\lambda)\equiv M(x,t,\lambda)$.
\subsubsection{  Singularities at $\pm\frac{1}{A_j}$.} 

Let $A^{(ij)}$ denote the elements of a $2\times 2$ matrix $A=\begin{pmatrix}
A^{(11)}&A^{(12)}\\A^{(21)}&A^{(22)}
\end{pmatrix}$.


\begin{proposition}\label{prop:sing-M} $M(x,t,\lambda)$ has the following behaviour at the branch points
\begin{equation}\label{sing}
M(x,t,\lambda)=\begin{cases}
\frac{\eul^{\frac{3\pi\ii}{4}}}{\nu_2^+(\lambda)}\begin{pmatrix}
0&\Upsilon_2\\0&\Lambda_2
\end{pmatrix}+\ord(1),\quad \lambda\to\frac{1}{A_2},\\
\frac{\ii}{\nu_2^-(\lambda)}\begin{pmatrix}
0&\Upsilon_2\\0&-\Lambda_2
\end{pmatrix}+\ord(1),\quad \lambda\to-\frac{1}{A_2},\\
\frac{c_+\eul^{\frac{3\pi\ii}{4}}}{\nu_1^+(\lambda)}\begin{pmatrix}
\Upsilon_1&0\\\Lambda_1&0
\end{pmatrix}+\ord(1),\quad \lambda\to\frac{1}{A_1},~ \lambda\in\mathbb{C}_+,\\
\frac{\overline{c_+}\eul^{\frac{3\pi\ii}{4}}}{\nu_1^+(\lambda)}\begin{pmatrix}
\Upsilon_1&0\\\Lambda_1&0
\end{pmatrix}+\ord(1),\quad \lambda\to\frac{1}{A_1},~ \lambda\in\mathbb{C}_-,\\
\frac{\overline{c_+}\ii}{\nu_1^-(\lambda)}\begin{pmatrix}
-\Upsilon_1&0\\\Lambda_1&0
\end{pmatrix}+\ord(1),\quad \lambda\to-\frac{1}{A_1},~ \lambda\in\mathbb{C}_+,\\
\frac{c_+\ii}{\nu_1^-(\lambda)}\begin{pmatrix}
-\Upsilon_1&0\\\Lambda_1&0
\end{pmatrix}+\ord(1),\quad \lambda\to-\frac{1}{A_1},~ \lambda\in\mathbb{C}_-,
\end{cases}
\end{equation}
where $\nu_j^\pm(\lambda)$ are defined in Proposition \ref{prop:sing_beh_D}, and $\Upsilon_j=-(2A_j)^{\frac{1}{4}}\alpha_j(x,t)+\frac{(a_j(x,t)+b_j(x,t))}{(2A_j)^{\frac{1}{4}}}$, $\Lambda_j=(2A_j)^{\frac{1}{4}}\alpha_j(x,t)+\frac{(a_j(x,t)+b_j(x,t))}{(2A_j)^{\frac{1}{4}}}$ with $\alpha_j(x,t),~a_j(x,t),~b_j(x,t)\in\D{R}$, $j=1,2$ as in Proposition \ref{prop:sing_beh}.

Moreover, $c_+(x,t)=0$ if $\beta_1(x,t)\neq 0$ and $c_+(x,t)=\frac{1}{\tilde s_{11}(x,t,\frac{1}{A_2})}$ if $\beta_1(x,t)= 0$, where $\beta_1(x,t)$ is defined in \eqref{beta_1}.
\end{proposition}
\begin{proof}
Combining Proposition \ref{prop:sing_beh} with Proposition \ref{prop:sing_beh_D} we get
\begin{align*}
    D_j^{-1}(\lambda)\tilde\Phi_j(x,t,\lambda)&=\frac{\eul^{\frac{3\pi\ii}{4}}}{\nu^+_j(\lambda)}\left(-(2A_j)^{\frac{1}{4}}\alpha_j\begin{pmatrix}
1&1\\-1&-1
\end{pmatrix}+\frac{a_j+b_j}{(2A_j)^{\frac{1}{4}}}\begin{pmatrix}
1&1\\1&1
\end{pmatrix}\right)+\ord\left((\lambda-\frac{1}{A_j})^{1/4}\right)
\end{align*}
as $\lambda\to\frac{1}{A_j}$, where $\alpha_j=\alpha_j(x,t)$, $a_j=a_j(x,t)$ and $b_j=b_j(x,t)$. 

First, consider the behaviour of $M$ near  $\frac{1}{A_2}$.
Since $D_1^{-1}(\lambda)\tilde\Phi_1^{(1)}(x,t,\lambda)$ is analytic at $\frac{1}{A_2}$, we have \[D_1^{-1}(\frac{1}{A_2})\tilde\Phi_1^{(1)}(x,t,\frac{1}{A_2})=\ii\begin{pmatrix}
a(x,t)\\c(x,t)
\end{pmatrix}\] 
with 
\[a(x,t)=\left|\sqrt{\frac{A_2+|\sqrt{A_2^2-A_1^2}|}{|\sqrt{A_2^2-A_1^2}|}}\right|\left(\frac{A_1}{A_2+|\sqrt{A_2^2-A_1^2}|}\tilde\Phi_1^{(11)}(x,t,\frac{1}{A_2})+\tilde\Phi_1^{(21)}(x,t,\frac{1}{A_2})\right)\] 
and 
\[c(x,t)=\left|\sqrt{\frac{A_2+|\sqrt{A_2^2-A_1^2}|}{|\sqrt{A_2^2-A_1^2}|}}\right|\left(\frac{A_1}{A_2+|\sqrt{A_2^2-A_1^2}|}\tilde\Phi_1^{(21)}(x,t,\frac{1}{A_2})+\tilde\Phi_1^{(11)}(x,t,\frac{1}{A_2})\right).\]
Then, using \eqref{til_scatcoeff-a}, we get the following expansion of $\tilde s_{11}(x,t,\lambda)$ at $\frac{1}{A_2}$:
\begin{equation*}
 \tilde s_{11}(x,t,\lambda)=\frac{\ii\eul^{\frac{3\pi\ii}{4}}}{\nu_2^+(\lambda)}\beta_2(x,t)+\ord(1), \quad \lambda\to\frac{1}{A_2}  
\end{equation*}
with $\beta_2(x,t)=\left((2A_2)^{\frac{1}{4}}\alpha_2(x,t)(a(x,t)+c(x,t))+\frac{(a_2(x,t)+b_2(x,t))(a(x,t)-c(x,t))}{(2A_2)^{\frac{1}{4}}}\right)$.

Notice that the symmetry \eqref{sym_tildephi_barbar_notSigma_i} implies that  $\tilde\Phi_1^{(11)}(x,t,\frac{1}{A_2})$ and $\tilde\Phi_1^{(21)}(x,t,\frac{1}{A_2})$ are real-valued and thus $a(x,t)\in\mathbb{R}$ and $c(x,t)\in\mathbb{R}$.

Recall the assumption $s_{11}(\frac{1}{A_2})\neq 0$, which  implies $\tilde s_{11}(\frac{1}{A_2})\neq 0$. Thus there are two possibilities: either $\beta_2(x,t)\neq 0$ or  $\beta_2(x,t)=0$ and $\tilde s_{11}(\frac{1}{A_2})=:\gamma\neq 0$. In the both cases,
\[M(x,t,\lambda)=\frac{\eul^{\frac{3\pi\ii}{4}}}{\nu_2^+(\lambda)}\begin{pmatrix}
0&-(2A_2)^{\frac{1}{4}}\alpha_2(x,t)+\frac{(a_2(x,t)+b_2(x,t))}{(2A_2)^{\frac{1}{4}}}\\0&(2A_2)^{\frac{1}{4}}\alpha_2(x,t)+\frac{(a_2(x,t)+b_2(x,t))}{(2A_2)^{\frac{1}{4}}}
\end{pmatrix}+\ord(1),\quad \lambda\to\frac{1}{A_2}.\]

Now consider the behaviour of $M$ as $\lambda$ approaches $\frac{1}{A_1}$ from the upper half-plane.
Since $D_2^{-1}(\lambda)\tilde\Phi_2^{(2)}(x,t,\lambda)$ has no singularity at $\frac{1}{A_1}$, we have \[D_2^{-1}(\frac{1}{A_1}_+)\tilde\Phi_2^{(2)}(x,t,\frac{1}{A_1}_+)=\begin{pmatrix}
b_+(x,t)\\d_+(x,t)
\end{pmatrix}\]
with  
\[b_+=\left|\sqrt{\frac{-\ii A_1-|\sqrt{A_2^2-A_1^2}|}{|\sqrt{A_2^2-A_1^2}|}}\right|\left(\frac{A_2}{A_1-\ii|\sqrt{A_2^2-A_1^2}|}\tilde\Phi_2^{(12)}(x,t,\frac{1}{A_1}_+)+\tilde\Phi_2^{(22)}(x,t,\frac{1}{A_1}_+)\right)\] 
and
\[d_+=\left|\sqrt{\frac{-\ii A_1-|\sqrt{A_2^2-A_1^2}|}{|\sqrt{A_2^2-A_1^2}|}}\right|\left(\frac{A_2}{A_1-\ii|\sqrt{A_2^2-A_1^2}|}\tilde\Phi_2^{(22)}(x,t,\frac{1}{A_1}_+)+\tilde\Phi_2^{(12)}(x,t,\frac{1}{A_1}_+)\right).\]
Then, using \eqref{til_scatcoeff-a}, we get the following expansion of $\tilde s_{11}(x,t,\lambda)$ at $\frac{1}{A_1}$ in the upper half-plane:
\begin{subequations}\label{s_11_sing}
\begin{equation}\label{s_11_sing_}
 \tilde s_{11}(x,t,\lambda)=\frac{\eul^{\frac{3\pi\ii}{4}}}{\nu_1^+(\lambda)}\beta_1(x,t)+\ord(1), \quad \lambda\to\frac{1}{A_1},\quad \lambda\in\mathbb{C}_+  
\end{equation}
with
\begin{equation}\label{beta_1}
\beta_1(x,t)=-(2A_2)^{\frac{1}{4}}\alpha_1(x,t)(b_+(x,t)+d_+(x,t))+\frac{(a_1(x,t)+b_1(x,t))(d_+(x,t)-b_+(x,t))}{(2A_1)^{\frac{1}{4}}}.    
\end{equation}
\end{subequations}
As above, we have two possibilities: either $\beta_1(x,t)\neq 0$ (generic case) or  $\beta_1(x,t)= 0$ and $\tilde s_{11}(\frac{1}{A_1}_+)=\gamma_1^+\neq 0$. This gives 
\[M(x,t,\lambda)=\frac{c_+\eul^{\frac{3\pi\ii}{4}}}{\nu_1^+(\lambda)}\begin{pmatrix}
-(2A_1)^{\frac{1}{4}}\alpha_1(x,t)+\frac{(a_1(x,t)+b_1(x,t))}{(2A_1)^{\frac{1}{4}}}&0\\(2A_1)^{\frac{1}{4}}\alpha_1(x,t)+\frac{(a_1(x,t)+b_1(x,t))}{(2A_1)^{\frac{1}{4}}}&0
\end{pmatrix}+\ord(1),\quad \lambda\to\frac{1}{A_1},~ \lambda\in\mathbb{C}_+,\]
where $c_+=0$ if $\beta_1(x,t)\neq 0$, and $c_+=\frac{1}{\tilde s_{11}(\frac{1}{A_1}_+)}$ if $\beta_1(x,t)= 0$.

The other the statements follow from the symmetry considerations.
\end{proof}

\begin{remark}\label{Rem 3.8.}
\begin{enumerate}    
        \item $\rho(\lambda)=\frac{\tilde s_{21}(\lambda_+)}{\tilde  s_{11}(\lambda_+)}\eul^{-2p_2(x,t,\lambda_+)}=\ord(1)$ as $\lambda\to\frac{1}{A_2}$. Indeed, in the proof of the Proposition \ref{prop:sing-M}, we have seen that $\tilde s_{11}(x,t,\lambda)=\frac{\ii\eul^{\frac{3\pi\ii}{4}}}{\nu_2^+(\lambda)}\beta_2(x,t)+\ord(1)$ as $\lambda\to\frac{1}{A_2}$. Analogously, due to \eqref{til_scatcoeff-c}, we have $\tilde s_{21}(x,t,\lambda)=-\frac{\ii\eul^{\frac{3\pi\ii}{4}}}{\nu_2^+(\lambda)}\beta_2(x,t)+\ord(1)$ as  $\lambda\to\frac{1}{A_2}$. Moreover, by our assumptions, $\tilde s_{11}(\frac{1}{A_2})\neq 0$, and hence the claim follows.

        \item $\rho(\lambda)=\ord(1)$ as $\lambda\to\frac{1}{A_1}$.
    Indeed, we already know that $\tilde s_{11}(\lambda)=\frac{\eul^{\frac{3\pi\ii}{4}}}{\nu_1^+(\lambda)}\beta_1(x,t)+\ord(1)$ as $\lambda\to\frac{1}{A_1}$, $\lambda\in\mathbb{C}_+$. Analogously, \eqref{til_scatcoeff-c} together with \eqref{sym_bar} implies that if $\beta_1\neq 0$ we have $\tilde s_{21}(\lambda)=\frac{\ii\eul^{\frac{3\pi\ii}{4}}}{\nu_1^+(\lambda)}\overline{\beta_1(x,t)}+\ord(1)$, $\lambda\to\frac{1}{A_1}$, $ \lambda\in\mathbb{C}_+$. Moreover, by our assumptions, $\tilde s_{11}(\frac{1}{A_1}_+)\neq 0$, and hence the claim follows.

    \end{enumerate}
\end{remark} 

\subsubsection{Residue conditions.}
By \eqref{scat_til}, zeros of $\tilde s_{11}(\lambda)$ coincide with zeros $s_{11}(\lambda)$;
hence, by
Proposition \ref{prop:simple-zeros}, they are real and simple. Moreover, the symmetry \eqref{sym_s11_minus_notSigma_2} implies that $-\lambda_k$ is a zero of $\tilde s_{11}(\lambda)$ together with $\lambda_k$; we will denote
the set of zeros of $s_{11}(\lambda)$   by $\{\lambda_k,-\lambda_k\}_1^n$,
where $\lambda_k\in(0,\frac{1}{A_2})$.
\begin{proposition} $M^{(1)}$ has simple poles at $\{\lambda_k,-\lambda_k\}_1^n$. Moreover,
\begin{equation}\label{res}
    \Res_{\pm\lambda_k}M^{(1)}(x,t,\lambda)=\frac{b_k}{s'_{11}(\lambda_k)}\eul^{2p_2(\lambda_k)}M^{(2)}(x,t,\pm\lambda_k),
\end{equation}

Moreover, $\frac{b_k}{s'_{11}(\lambda_k)}\eul^{2p_2(\lambda_k)}\in\mathbb{R}$.
\end{proposition}

\begin{proof} Recall that $\Phi_1^{(1)}(\lambda_k)=b_k\Phi_2^{(2)}(\lambda_k)$ with $b_k=b(\lambda_k)\in\mathbb{R}$ due the symmetry \eqref{sym_phi_barbar-notin_Sigma_i}. Then $(D_1^{-1}\tilde\Phi_1^{(1)})(\lambda_k)$ and $(D_2^{-1}\tilde\Phi_2^{(2)})(\lambda_k)$ are related as
\[
 \frac{(D_1^{-1}\tilde\Phi_1^{(1)})(\lambda_k)}{s_{11}(\lambda_k)\eul^{p_1(\lambda_k)-p_2(\lambda_k)}}=\frac{b_k}{s_{11}(\lambda_k)}\eul^{2p_2(\lambda_k)}(D_2^{-1}\tilde\Phi_2^{(2)})(\lambda_k),
\]
and hence \eqref{res} follows. Moreover, differentiating \eqref{sym_s11_barbar} and using the fact that $\lambda_k\in\mathbb{R}$, we get $s'_{11}(\lambda_k)\in\mathbb{R}$, and thus $\frac{b_k}{s'_{11}(\lambda_k)}\eul^{2p_2(\lambda_k)}\in\mathbb{R}$.


Differentiating \eqref{sym_s11_minus_notSigma_2}, we get $s'_{11}(\lambda_k)=-s'_{11}(-\lambda_k)$. On the other hand, \eqref{sym_Phi_minus_notSigma_2} implies that $b(-\lambda_k)=-b(\lambda_k)$. Combining these facts, we obtain \eqref{res}
with the minus sign.
\end{proof}

\begin{remark}
In terms of  $\tilde M$ \eqref{tilde-M}, the  residue conditions take the following
form:
\begin{subequations}\label{tildeM_res}
\begin{alignat}{4}
    &\tilde M^{(1)}(x,t,\lambda)=\frac{1}{\lambda-\lambda_{k}}
    \frac{b_k}{s'_{11}(\lambda_k)}\eul^{2p_2(\lambda_{k})}\tilde M^{(2)}(x,t,\lambda_{k+})+\ord(1),~\lambda\to\lambda_k,~\lambda\in\mathbb{C}_+,\\
    &\tilde M^{(2)}(x,t,\lambda)=\frac{1}{\lambda-\lambda_{k}} 
    \frac{b_k}{s'_{11}(\lambda_k)}\eul^{2p_2(\lambda_{k})}\tilde M^{(1)}(x,t,\lambda_{k-})+\ord(1),~\lambda\to\lambda_k,~\lambda\in\mathbb{C}_-.
\end{alignat}
\end{subequations}

\end{remark}

\subsubsection{RH problem parametrized by $\BS{(x,t)}$.}\label{sec:RH-x-t}
In the framework of the Riemann--Hilbert approach to nonlinear evolution equations, one  interprets the jump relation, normalization condition, singularity conditions, and residue conditions as a Riemann--Hilbert problem, with the jump matrix and residue parameters determined by the initial data for the nonlinear problem in question.
The considerations above imply that $M(x,t,\lambda)$ can be characterized as the solution 
of the  following Riemann-Hilbert problem:

Find a $2\times 2$ meromorphic matrix $M(x,t,\lambda)$ that satisfies the following conditions:
\begin{enumerate}[\textbullet]
\item 
\emph{Jump} condition \eqref{jump_}.
\item
\emph{Normalization} condition \eqref{M-norm}.
\item
\emph{Singularity} conditions: the singularities of $M(x,t,\lambda)$ at $\pm\frac{1}{A_j}$ are of order not bigger than $\frac{1}{4}$.

\item
\emph{Residue} conditions (if any): 
given $\accol{\lambda_k,  \kappa_k}_1^N$ with $\lambda_k\in (0,\frac{1}{A_2})$
and $\kappa_k\in\mathbb{R}\setminus\{0\}$, 
$M^{(1)}(x,t,\lambda)$ has simple poles at $\accol{\lambda_k, -\lambda_k}_1^N$, with the residues  satisfying the equations

\begin{equation}\label{res-M}
    \Res_{\pm\lambda_k}M^{(1)}(x,t,\lambda)=\kappa_k\eul^{2p_2(\lambda_k)}M^{(2)}(x,t,\pm\lambda_k).
\end{equation}

\end{enumerate}
\begin{remark}
The solution of the RH problem above, if exists, satisfies the following properties:
\begin{enumerate}
\item
 $\det M\equiv 1$ (follows from the fact that $\det J\equiv1$).
\item 
\emph{Symmetries} 
\begin{subequations}\label{sym-M_}
\begin{alignat}{4}
M(-\lambda)&=-\sigma_3 M(\lambda)\sigma_3,\qquad \overline{M(\overline{\lambda})}=-M(\lambda),\qquad \lambda\in\mathbb{C}\setminus\Sigma_2,\label{sym-M_a}\\
M((-\lambda)_-)&=-\sigma_3 M(\lambda_+)\sigma_3,\qquad \overline{M(\lambda_-)}=-M(\lambda_+),\qquad \lambda\in\dot\Sigma_1.
\end{alignat}
\end{subequations}
where $M(\lambda)\equiv M(x,t,\lambda)$
(follows from the respective symmetries of the jump matrix and the residue conditions,
assuming the uniqueness of the solution).
\end{enumerate}
\end{remark}
\begin{remark}
We do not need to specify the singularities at the branch points $\pm\frac{1}{A_j}$ in order to formulate RH problem. It is enough to require them to be of order not bigger than $\frac{1}{4}$.
\end{remark}

As for other Camassa--Holm-type equations, a principal drawback of the RH formalism
presented above is that the jump condition \eqref{jump_} involves not only
the scattering functions uniquely determined by the initial data for problem 
\eqref{mCH1-ic}, but the solution itself, via $p_2(x,t,\lambda)$
involving $m(x,t)$ \eqref{p_i}.
In order to have the data for a RH problem to 
be  explicitly determined by the initial data only,  we introduce the space variable $y(x,t)\coloneqq x-\frac{1}{A_2}\int_x^{+\infty}( m(\xi,t)-A_2)\dd\xi-A_2^2 t$,
which will play the role of a parameter (together with $t$) for the RH problem,
see Section \ref{sec3-3} below.

In order to determine an efficient way for retrieving the solution of the mCH equation from the solution of the RH problem, we will use   the behavior of the Jost solutions of the Lax pair equations evaluated at $\lambda=0$, for which the $x$-equation \eqref{Lax-x} of the Lax pair becomes trivial (independent of the solution of the mCH equation).


\subsection{Eigenfunctions near $0$.}\label{Eigen_0} 
In the case of the Camassa--Holm equation \cite{BS08} as well as other CH-type nonlinear integrable equations studied so far, see, e.g., \cites{BS15}, the analysis of the behavior of the respective Jost solutions at a dedicated point in the complex plane of the spectral parameter (in our case, at $\lambda=0$) requires a dedicated gauge transformation of the Lax pair equations.

It is remarkable that in the case of the mCH equation, in order to control the 
behavior of the eigenfunctions at $\lambda=0$, we don't need to introduce an additional transformation; all we need is to regroup the terms in the Lax pair \eqref{Lax-1}.

Namely, we rewrite \eqref{Lax-1} as follows:

\begin{subequations}\label{Lax-1_0}
\begin{equation}\label{Lax-2-x}
\hat\Phi_{jx}+\frac{\ii A_j k_j(\lambda)}{2}\sigma_3\hat\Phi_j=\hat U^0_j \hat\Phi_j,
\end{equation}
where $\hat U^0_j\equiv\hat U^0_j(x,t,\lambda)$ is given by
\begin{equation}\label{U^0_i-hat}
\hat U^0_j=\frac{(m-A_j)}{2}\frac{\lambda}{\ii k_j(\lambda)}
\begin{pmatrix}
\lambda & \frac{1}{A_j} \\
-\frac{1}{A_j} & -\lambda \\
\end{pmatrix},
\end{equation}
and
\begin{equation}\label{Lax-2-t}
\hat\Phi_{jt} +\ii A_j k_j(\lambda)
\left(-\frac{A_j^2}{2} -\frac{1}{\lambda^2}\right)\sigma_3\hat\Phi_j= \hat V^0_j \hat\Phi_j,
\end{equation}
where $\hat V^0_j\equiv\hat V^0_j(x,t,\lambda)$ is given by
\begin{equation}\label{hat-V^0_i}
\begin{aligned}
\hat V^0_j&=\hat V_j+\ii A_j k_j(\lambda)\left(\frac{(u^2-u_x^2)m}{2A_j}-\frac{A_j^2}{2}\right) \sigma_3.
\end{aligned}
\end{equation}
\end{subequations}

Further, introduce (compare with \eqref{p_i})
\begin{equation}\label{p_i_0}
p^0_j(x,t,\lambda)\coloneqq\frac{\ii A_j k_j(\lambda)}{2}\left(x-2\left(\frac{A_j^2}{2} +\frac{1}{\lambda^2}\right)t\right).
\end{equation}
Then, introducing  $Q^0_j\coloneqq p^0_j\sigma_3$ and $\widetilde\Phi^0_j\coloneqq\hat\Phi_j \eul^{Q_j^0}$,  equations \eqref{Lax-2-x} and \eqref{Lax-2-t} reduce to \begin{equation}\label{comsys-0}
\begin{cases}
\widetilde\Phi^0_{jx}+[Q^0_{jx},\widetilde\Phi_j^0]=\hat U^0_j\widetilde\Phi_j^0,&\\
\widetilde\Phi^0_{jt}+[Q^0_{jt},\widetilde\Phi_j^0]=\hat V^j_0\widetilde\Phi_j^0.&
\end{cases}
\end{equation}
Define the Jost solutions $\widetilde\Phi_{j}^0$ of \eqref{comsys-0} as the solutions of the integral equations
\begin{equation}\label{inteq0_lam}
\widetilde\Phi_{j}^0(x,t,\lambda)=I+\int_{(-1)^j\infty}^x\eul^{\frac{-\ii A_j k_j(\lambda)}{2}(x-\xi)\sigma_3}\hat U_j^0(\xi,t,\lambda)\widetilde\Phi_j^{0}(\xi,t,\lambda)\eul^{\frac{\ii A_j k_j(\lambda)}{2}(x-\xi)\sigma_33}\dd\xi.
\end{equation}

Further, defining $\hat\Phi_{j}^0\coloneqq\widetilde\Phi^0_{j}\eul^{-p_j^0 \sigma_3}$, we observe that $\hat\Phi_{j}^0(x,t,\lambda)$ and $\hat\Phi_j(x,t,\lambda)$ satisfy the same differential equations \eqref{Lax-1} and thus they are related by matrices $C_j(\lambda)$ independent of $x$ and $t$:
\[
\hat\Phi_{j}=\hat\Phi_j^{0}C_j(\lambda).
\]
Consequently,
\begin{equation}\label{inf_0_rel}
\widetilde\Phi_{j}(x,t,\lambda)=\widetilde\Phi_j^{0}(x,t,\lambda)\eul^{-p_j^0(x,t,\lambda) \sigma_3}C_j(\lambda)\eul^{p_j(x,t,\lambda)\sigma_3}.
\end{equation}
Since $p_j(x,t,\lambda)-p_j^0(x,t,\lambda)=\frac{\ii k_j(\lambda)}{2
}\int_x^{(-1)^j\infty}( m(\xi,t)-A_j)\dd\xi$ and 
\[
\widetilde\Phi_{j}(x,t,\lambda)=\widetilde\Phi_j^{0}(x,t,\lambda)\eul^{\frac{\ii k_j(\lambda)}{2
}\int^x_{(-1)^j\infty}( m(\xi,t)-A_j)\dd\xi \sigma_3},
\]
passing to the limits
$x\to (-1)^j\infty$, we get  $C_j(\lambda)=I$.

Noticing that  $\hat U_j^0(x,t,0)\equiv 0$, it follows from \eqref{inteq0_lam} that $\widetilde\Phi_j^{0}(x,t,0)\equiv I$ and thus $\widetilde\Phi_j(x,t,0)=\eul^{-\frac{1}{2A_j
}\int^x_{(-1)^j\infty}( m(\xi,t)-A_j)\dd\xi \sigma_3}$.
Combining this with $D_j^{-1}(0)=\begin{pmatrix}
0&i\\i&0
\end{pmatrix}$ gives
\[
(D_j^{-1}\widetilde\Phi_j)(x,t,0)=\ii\begin{pmatrix}
0&\eul^{\frac{1}{2A_j
}\int^x_{(-1)^j\infty}( m(\xi,t)-A_j)\dd\xi }\\
\eul^{-\frac{1}{2A_j
}\int^x_{(-1)^j\infty}( m(\xi,t)-A_j)\dd\xi}&0
\end{pmatrix}
\]
Consequently,
\[
\tilde s_{11}(0)=\eul^{-\frac{1}{2A_1
}\int^x_{-\infty}( m(\xi,t)-A_1)\dd\xi-\frac{1}{2A_2
}\int_x^{\infty}( m(\xi,t)-A_2)\dd\xi}
\]
(hence $\tilde s_{11}(0)\neq0$) and
\begin{equation}\label{M(0)}
M(x,t,0) = \ii\begin{pmatrix}
0&\eul^{-\frac{1}{2A_2
}\int_x^{\infty}( m(\xi,t)-A_2)\dd\xi }\\
\eul^{\frac{1}{2A_2
}\int_x^{\infty}( m(\xi,t)-A_2)\dd\xi}&0
\end{pmatrix}.
\end{equation}

\begin{remark}\label{rem3-1}
Considering $M(x,t,\lambda)$ as the solution of the RH problem
in Section \ref{sec:RH-x-t}, 
the matrix structure of $M(x,t,0)$ as in \eqref{M(0)}, i.e.,
\begin{equation}\label{M(pmi-a)}
M(x,t,0) =
\ii\begin{pmatrix}
0 & a_1(x,t) \\
a_1^{-1}(x,t) & 0
\end{pmatrix}
\end{equation}
with some $a(x,t)\in\D{R}$,  
 which follows from the symmetry properties \eqref{sym-M_a}  of the solution
 taking into account that $\det M\equiv 1$ (provided the solution is unique).  
\end{remark}

In order to extract the solution of the mCH equation from the solution of the associated RH problem, it turns to be useful to  find the next term in the expansion of $M(x,t,\lambda)$ at $\lambda=0$.

First, expanding $D_j^{-1}(\lambda)$ near $0$, we have 
\[D_j^{-1}(\lambda)=\begin{pmatrix}
0&i\\i&0
\end{pmatrix}+\lambda\begin{pmatrix}
\ii\frac{A_j}{2}&0\\0&\ii\frac{A_j}{2}
\end{pmatrix}+\ord(\lambda^2).\]
On the other hand,
$\eul^{\frac{\ii k_j(\lambda)}{2
}\int^x_{(-1)^j\infty}( m(\xi,t)-A_j)\dd\xi \sigma_3}=\eul^{-\frac{ 1}{2A_j
}\int^x_{(-1)^j\infty}( m(\xi,t)-A_j)\dd\xi \sigma_3}+\ord(\lambda^2),~\lambda\to0.$
Then, expanding $\widetilde\Phi_{j}^0(x,t,\lambda)$ at $0$ using the Neumann series, we have 
\[\widetilde\Phi_{j}^0(x,t,\lambda)=I+\lambda\begin{pmatrix}
0&-\int_{(-1)^j\infty}^x e^{x-\xi}\frac{m-A_j}{2}\dd \xi\\
\int_{(-1)^j\infty}^x e^{-(x-\xi)}\frac{m-A_j}{2}\dd \xi&0
\end{pmatrix}+\ord(\lambda^2).\]
In particular,
\[
\tilde s_{11}(\lambda)=\eul^{-\frac{1}{2A_1
}\int^x_{-\infty}( m(\xi,t)-A_1)\dd\xi-\frac{1}{2A_2
}\int_x^{\infty}( m(\xi,t)-A_2)\dd\xi}+\ord(\lambda^2).
\]
Finally, we have 
\begin{equation}\label{x-expan} 
M(x,t,\lambda)=\ii\begin{pmatrix}
0& a_1(x,t)\\ a_1^{-1}(x,t)&0
\end{pmatrix}+\ii\lambda\begin{pmatrix}
a_2(x,t)&0\\0&a_3(x,t)
\end{pmatrix}+\ord(\lambda^2),
\end{equation}
where 
\begin{subequations}\label{a_j}
    \begin{alignat}{4}
    a_1(x,t) &= \eul^{-\frac{1}{2A_2
}\int_x^{\infty}( m(\xi,t)-A_2)\dd\xi},\\
a_2(x,t) &= (\int_{-\infty}^x e^{-(x-\xi)}\frac{m-A_1}{2}\dd\xi+\frac{A_1}{2})\eul^{\frac{1}{2A_2
}\int_x^{\infty}( m(\xi,t)-A_2)\dd\xi},\\
a_3(x,t) &= (\int^{\infty}_x e^{(x-\xi)}\frac{m-A_2}{2}\dd\xi+\frac{A_2}{2})\eul^{-\frac{1}{2A_2
}\int_x^{\infty}( m(\xi,t)-A_2)\dd\xi}.
    \end{alignat}
\end{subequations}

Notice that the matrix structure of terms in the r.h.s. of \eqref{x-expan} is consistent with 
 the symmetry properties \eqref{sym-M_a}  of $M$. 
 \begin{proposition}\label{prop-recover-x}
 $u(x,t)$ and $u_x(x,t)$ can be algebraically expressed in terms of the coefficients
 $a_j(x,t)$, $j=1,3$ in the development \eqref{x-expan} of $M(x,t,\lambda)$ as follows:
 \begin{subequations}\label{uu_x}
    \begin{alignat}{4}
    u(x,t) &= a_1(x,t)a_2(x,t)+ a_1^{-1}(x,t) a_3(x,t),\\
    \label{uu_x-b}
    u_x(x,t) & = -a_1(x,t)a_2(x,t)+ a_1^{-1}(x,t)a_3(x,t).
    \end{alignat}
\end{subequations}
 \end{proposition}
\begin{proof}
Introduce $v(x,t) := a_1(x,t)a_2(x,t)+ a_1^{-1}(x,t) a_3(x,t)$. Using \eqref{a_j}
it follows that 
\begin{equation}\label{v-a}
v(x,t) = \frac{A_1+A_2}{2}+\int_{-\infty}^x e^{-(x-\xi)}\frac{m-A_1}{2}\dd\xi+\int^{\infty}_x e^{(x-\xi)}\frac{m-A_2}{2}\dd\xi
\end{equation}
and thus, differentiating w.r.t. $x$, 
\begin{equation}\label{v_x-a}
v_x(x,t) = \frac{A_2-A_1}{2}-\int_{-\infty}^x e^{-(x-\xi)}\frac{m-A_1}{2}\dd\xi+\int^{\infty}_x e^{(x-\xi)}\frac{m-A_2}{2}\dd\xi.
\end{equation}
Since we assume that $\lim_{x\to(-1)^j\infty}m(x,t)=A_i$, from \eqref{v-a}
it follows that $v-v_{xx}=m$ and that 
\[
\lim_{x\to(-1)^j\infty}v(x,t)=A_i, \qquad 
\lim_{x\to(-1)^i\infty}v_x(x,t)=0;
\]
thus $v\equiv u$. Finally, we notice that the expression in the r.h.s. of \eqref{v_x-a}
can be written as the r.h.s. of \eqref{uu_x-b} taking into account \eqref{a_j}.

\end{proof}


\subsection{RH problem in the $\BS{(y,t)}$ scale}\label{sec3-3}
As we already mentioned, 
 the jump condition \eqref{jump_} involves not only
the scattering functions uniquely determined by the initial data for problem 
\eqref{mCH1-ic}, but the solution itself, via $m(x,t)$, which 
enters the definition of $p_2(x,t,\lambda)$
 \eqref{p_i}.
In order to have the data for the RH problem to 
be  explicitly determined by the initial data only,  we introduce the new space variable  $y(x,t)$ by 
\begin{equation}\label{shkala}
    y(x,t)=x-\frac{1}{A_2}\int_x^{+\infty}(m(\xi,t)-A_2)\dd\xi-A_2^2 t,
\end{equation} 
Then, introducing $\hat M(y,t,\lambda)$ so that $M(x,t,\lambda)=\hat M (y(x,t),t,\lambda)$, the dependence of the jump matrix in  \eqref{jump_} 
on $y$ and $t$ as parameters becomes explicit: the jump condition for 
$\hat M(y,t,\lambda)$ has the form
\begin{subequations} \label{Jp-y}
\begin{equation}\label{jump-y}
\hat M^+(y,t,\lambda)=\hat M^-(y,t,\lambda)\hat J(y,t,\lambda),\quad\lambda\in \dot\Sigma_1\cup\dot\Sigma_0.
\end{equation}
Here
\begin{equation}\label{J-J0}
\hat J(y,t,\lambda)\coloneqq\begin{pmatrix}
0&\ii\\
\ii&0
\end{pmatrix}\begin{pmatrix}
\eul^{-\hat p_2(y,t,\lambda_+)}&0\\
0&\eul^{\hat p_2(y,t,\lambda_+)}
\end{pmatrix}J_0(\lambda)
\begin{pmatrix}
\eul^{\hat p_2(y,t,\lambda_+)}&0\\
0&\eul^{-\hat p_2(y,t,\lambda_+)}
\end{pmatrix},
\end{equation}
where $J_0(\lambda)$ is defined by \eqref{jump_J0} and
$p_2$ is explicitly given in terms of $y$ and $t$:
\begin{equation}\label{p-y}
\hat p_2(y,t,\lambda) \coloneqq \frac{\ii A_2 k_2(\lambda)}{2}\left(y-\frac{2t}{\lambda^2}\right).
\end{equation}
\end{subequations}

Similarly,  the residue conditions \eqref{res-M} become explicit 
as well:

\begin{equation}\label{res-M-hat}
\Res_{\pm\lambda_k}\hat M^{(1)}(y,t,\lambda)=\kappa_k\eul^{2\hat p_2(y,t,\lambda_k)}\hat M^{(2)}(y,t,\pm\lambda_k),
\end{equation}
with $\kappa_k=\frac{b_k}{s'_{11}(\lambda_k)}$.

Noticing that the normalization condition \eqref{M-norm} and the singularity conditions at $\lambda=\pm \frac{1}{A_j}$ hold in the new scale $(y,t)$, we arrive at the basic RH problem
characterizing problem \eqref{mCH-1}.

\begin{rh-pb*}
Given $\rho(\lambda)$ for $\lambda\in \dot\Sigma_1\cup\dot\Sigma_0$, and $\accol{\lambda_k,\kappa_k}_1^N$ with $\lambda_k\in(0,\frac{1}{A_2})$ and $\kappa_k\in\mathbb{R}\setminus\{0\}$, 
associated with the initial data $u_0(x)$ in \eqref{mCH1-ic},
find a piece-wise (w.r.t.~$\dot\Sigma_2$) meromorphic, $2\times 2$-matrix valued function $\hat M(y,t,\lambda)$ satisfying the following conditions:
\begin{enumerate}[\textbullet]
\item
Jump condition \eqref{Jp-y} across $\dot\Sigma_1\cup\dot\Sigma_0$ (with $J_0(\lambda)$ defined by \eqref{jump_J0}).
\item
Residue conditions \eqref{res-M-hat}.
\item Normalization condition:
\begin{equation}\label{norm-hat}
\hat M(y,t,\lambda)=\begin{cases}\sqrt{\frac{1}{2}}
\begin{pmatrix}
-1&\ii\\
\ii&-1
\end{pmatrix}+\ord(\frac{1}{\lambda}), \quad \lambda\to\infty,~ \lambda\in\mathbb{C}^+,\\
\sqrt{\frac{1}{2}}
\begin{pmatrix}
1&\ii\\
\ii&1
\end{pmatrix}+\ord(\frac{1}{\lambda}), \quad \lambda\to\infty,~ \lambda\in\mathbb{C}^-.
\end{cases}
\end{equation}

\item
Singularity conditions: the singularities of $
\hat M(y,t,\lambda)$ at $\pm\frac{1}{A_j}$ are of order not bigger than $\frac{1}{4}$.
\end{enumerate}
\end{rh-pb*}
Evaluating the solution of this problem as $\lambda\to 0$,
we are able to present the solution $u$ to the initial value problem 
\eqref{mCH1-ic} in a parametric form, see below. 
As for the data for the RH problem, the scattering matrix 
$s(\lambda)$ (and hence $s_{11}(\lambda)$, $s_{21}(\lambda)$, and $\rho(\lambda)$)
as well as the discrete data $\left\{\lambda_k, \kappa_k\right\}_1^n$
are determined by $u_0(x)$ via the solutions of \eqref{inteq} considered for $t=0$.


\ifshort
The uniqueness of the solution of the basic RH problem follows
using standard arguments based on the application of Liouville's theorem
to the ratio $\hat M_1(\hat M_2)^{-1}$ of two 
potential solutions, $\hat M_1$ and $\hat M_2$. Particularly, the 
singularity condition implies that the possible singularities
of $\hat M_1(\hat M_2)^{-1}$ are of order no bigger that $1/2$ 
and that these singularities, being isolated, are removable.
\fi

The uniqueness, in particular, implies the symmetries
\begin{subequations}\label{sym-M-hat}
\begin{alignat}{4}
\hat M(-\lambda)&=-\sigma_3 \hat M(\lambda)\sigma_3,\qquad \overline{\hat M(\overline{\lambda})}=-M(\lambda),\qquad \lambda\in\mathbb{C}\setminus\Sigma_2,\\
\hat M((-\lambda)_-)&=-\sigma_3 \hat M(\lambda_+)\sigma_3,\qquad \overline{\hat M(\lambda_-)}=-\hat M(\lambda_+),\qquad \lambda\in\dot\Sigma_1.
\end{alignat}
\end{subequations}
where $\hat M(\lambda)\equiv \hat M(y,t,\lambda)$, 
which follows from the corresponding symmetries of $\hat J(y,t,\lambda)$.

\subsection{Recovering $\BS{u(x,t)}$ from the solution of the basic RH problem}\label{sec:recover}

Comparing the RH problem \eqref{jump_}, \eqref{M-norm}, \eqref{res-M} parametrized by 
$x$ and $t$ with the RH problem \eqref{Jp-y}--\eqref{norm-hat} parametrized by 
$y$ and $t$ and using \eqref{a_j}--\eqref{shkala} we arrive at our main representation result.

\begin{theorem}\label{main}
Assume that $u(x,t)$ is the solution of the Cauchy problem \eqref{mCH1-ic} and let  $\hat M(y,t,x)$ be the solution of the associated RH problem \eqref{Jp-y}--\eqref{norm-hat}, whose data are determined by $u_0(x)$.
Let  
\begin{equation}\label{M-hat-expand}
\hat M(y,t,\lambda)=\ii\begin{pmatrix}
0& \hat a_1(y,t)\\\hat a_1^{-1}(y,t)&0
\end{pmatrix}+\ii\lambda\begin{pmatrix}
\hat a_2(y,t)&0\\0&\hat a_3(y,t)
\end{pmatrix}+\ord(\lambda^2)
\end{equation}
be the development of  $\hat M(y,t,x)$ at $\lambda=0$. 
Then the solution $u(x,t)$ of the Cauchy problem \eqref{mCH1-ic} 
can be expressed, in a parametric form, in terms of $\hat a_j(y,t)$, $j=1,2,3$:
$u(x,t)=\hat u(y(x,t),t)$, where
\begin{subequations}\label{recover-2}
\begin{align}\label{u_(y,t)}
&\hat u(y,t)=\hat a_1(y,t)\hat a_2(y,t)+\hat a_1^{-1}(y,t)\hat a_3(y,t),\\
&x(y,t)=y-2\ln\hat a_1(y,t)+A_2^2 t.
\label{x(y,t)-2}
\end{align}
Additionally, $\hat u_x(y,t)$ can also be algebraically expressed in terms of $\hat a_j(y,t)$, $j=1,2,3$: 
$u_x(x,t)=\hat u_x(y(x,t),t)$, where
\begin{equation}
\label{u_x(y,t)-2}
\hat u_x(y,t)=-\hat a_1(y,t)\hat a_2(y,t)+\hat a_1^{-1}(y,t)\hat a_3(y,t).
\end{equation}
\end{subequations}
\end{theorem}

Alternatively, one can express $\hat u_x(y,t)$ 
in terms of the first term in \eqref{M-hat-expand} only.
The price to pay is that this expression involves the derivatives of this term. 
\begin{proposition}\label{prop-recover_1}
The $x$-derivative of the solution $u(x,t)$ of the Cauchy problem \eqref{mCH1-ic} has the parametric representation
\begin{subequations}\label{recover-1}
\begin{align}\label{u_x(y,t)}
&\hat u_x(y,t)=-\frac{1}{A_2}\partial_{ty}\ln \hat a_1(y,t),\\
\label{x(y,t)}
&x(y,t)=y-2\ln\hat a_1(y,t)+A_2^2 t.
\end{align}
\end{subequations}

\end{proposition}
\begin{proof}
Differentiating  the identity $x(y(x,t),t)=x$ w.r.t.~$t$ gives
\begin{equation}\label{dt-0}
0=\frac{d}{dt}\left(x(y(x,t),t)\right)=x_y(y,t)y_t(x,t)+x_t(y,t).
\end{equation}
From \eqref{shkala} it follows that
\begin{equation}\label{x_y}
x_y(y,t)=\frac{A_2}{\hat m(y,t)},
\end{equation}
where $\hat m(y,t)= m(x(y,t),t)$,
and 
\[
y_t(x,t)=-\frac{1}{A_2}( u^2- u_x^2)m.
\]
Substituting this and \eqref{x_y} into \eqref{dt-0} we obtain
\begin{equation}\label{x_t}
x_t(y,t)=\hat u^2(y,t)-\hat u_x^2(y,t).
\end{equation}
Further, differentiating \eqref{x_t} w.r.t.~$y$ we get
\begin{equation}\label{x_yt}
x_{ty}(y,t)=(\hat u^2(y,t)-\hat u_x^2(y,t))_x x_y(y,t)=2A_2\hat u_x(y,t)
\end{equation}
and thus 
\[
u_x(x(y,t),t)\equiv\hat u_x(y,t)=\frac{1}{2A_2}\partial_{ty}x(y,t)
=-\frac{1}{A_2}\partial_{ty}\ln \hat a_1(y,t).
\]
\end{proof}

\section{Concluding remarks}
We have presented the Riemann-Hilbert problem approach for the modified Camassa--Holm equation on the line with step-like boundary conditions. In the proposed formalism, we have taken the branch cut of $k_j(\lambda)$ along the half-lines $\Sigma_j$ (outer cuts), 
which is convenient since we extract the solution of the mCH equation exploiting 
the development of the solution of the RH problem at a point laying in the domain of analyticity. 
Notice that it is possible to formulate RH problem taking the  branch cut of $k_j(\lambda)$ to be the  segments $(-\frac{1}{A_j},\frac{1}{A_j})$ (inner cuts). In the case with inner cuts, the properties of Jost solutions are more conventional (two of the columns are analytic in the upper half-plane and other two in the lower half-plane), but, on the other hand, possible eigenvalues are located on 
the jump. 

The present paper is focused on the representation results while assuming the existence 
of a solution of problem \eqref{mCH1-ic} in certain functional classes.
To the best of our knowledge,  the question of existence  is still open. One of the ways to answering it is to appeal to functional analytic PDE techniques to obtain well-posedness
in appropriate functional classes.  However, very little is known for the cases of 
nonzero boundary conditions, particularly, for  backgrounds having different behavior at different 
infinities. Since 1980s, existence problems for integrable nonlinear PDE with step-like initial
conditions have been addressed using the classical Inverse Scattering Transform method \cite{K86}.
A more recent progress in this direction (in the case of the Korteweg-de Vries equation) has been reported in \cites{E09, E11, GR19} (see also \cite{E22}).  
Another way to show existence is to infer it from the RH problem formalism (see, e.g., \cite{FLQ21}
for the case of defocusing nonlinear Schr\"odinger equation), where a key point consists in establishing a solution of the associated RH problem and controlling its behavior w.r.t.
the spatial parameter.
For Camassa-Holm-type equations, where the RH problem formalism involves the change of the spatial 
variable, 
it is natural to study the existence of solution in both $(x,t)$ and $(y,t)$ scales. 
More precisely, the solvability problem splits into two problems: (i) the solvability of the RH
problem parametrized by $(y,t)$ and (ii) the bijectivity of the change of the spatial variable.
Particularly, it is possible that it is the change of variables that can be responsible of 
the wave breaking \cites{BSZ17,BKS20}. The solvability problem for problem \eqref{mCH1-ic}
in the current setting will be addressed elsewhere. 

Another interesting and important problem that 
can be addressed using the developed approach is the 
 investigation of the large-time behavior of the solutions of the Cauchy problem \eqref{mCH1-ic} adapting the nonlinear steepest descent method.  

\section{Acknowledgement.} 
IK acknowledges partial support of the National Academy of Sciences of Ukraine under Grant No. 0121U111968.


\appendix

\section{Sign-preserving property of $m$}\label{app:A}

Assume that $u(x,t)-A_1\in H^{3}(-\infty, a)$ and $u(x,t)-A_2\in H^{3}(a,\infty)$ 
for any real $a$ and for any $t\in (0,T)$, where $T\le +\infty$ is the maximal existing time.
Then Morrey's inequality   implies that 
$(mu_x)(s,x)$ is uniformly bounded for $0<s<t<T$, $x\in \mathbb{R}$.

Consider the Cauchy problem for $q(t,x)$:
\begin{subequations}\label{CP}
\begin{alignat}{4}
&\frac{\dd q}{\dd t}=(u^2-u_x^2)(q(t,x),t),\quad t\in (0,T), ~x\in\mathbb{R},\\ \label{CP-ic}
&q(0,x)=x,\quad x\in\mathbb{R},
\end{alignat}
\end{subequations}
where $u(x,t)$ solves \eqref{mCH1-ic}.
Differentiating  \eqref{CP} with respect to $x$ leads to  
\begin{subequations}\label{difCP}
\begin{alignat}{4}
&\frac{\dd}{\dd t} q_x(t,x)=(2mu_x)(q(t,x),t)q_x(t,x),\\\label{difCP-ic}
&q_x(0,x)=1,\quad x\in\mathbb{R}.
\end{alignat}
\end{subequations}
It follows that 
\begin{equation}\label{q_x}
q_x(t,x)=\eul^{2\int_{0}^{t}(mu_x)(q(s,x),s)\dd s}>0  
\end{equation}
and, moreover, 
\begin{equation}\label{q_x-est}
   \eul^{k(t)}\leq q_x(t,x)\leq \eul^{K(t)},\quad t\in [0,T)
\end{equation}
for some $k(t)$ and $K(t)$.

Now observe that  from \eqref{mCH-1} and \eqref{CP} it follows that
 $\frac{\dd}{\dd t} \left[m(q(t,x),t)q_x(t,x)\right]=0$.
 Indeed,
 \begin{align*}
   \frac{d}{dt} & \left[m(q(t,x),t)q_x(t,x) \right] \\
  &  = \left[m_t(q(t,x),t)+m_x(q(t,x),t)q_t(t,x) \right](q(t,x),t)q_x(t,x) + m(q(t,x),t)q_{tx}(t,x) \\
  & = \left[-(u^2-u_x^2)_xm-(u^2-u_x^2)m_x +m_x(u^2-u_x^2) \right](q(t,x),t)q_x(t,x)\\ 
  & +2(m^2u_x)(q(t,x),t)q_x(t,x) 
 =0.
  \end{align*}
 Thus, due to \eqref{CP-ic} and \eqref{difCP-ic}, we have
\begin{equation*}
    m(t,q(t,x))q_x(t,x)=m(0,q(0,x))q_x(0,x)=m(0,x).
\end{equation*}
Hence, if $m(x,0)>0$, then  $m(q(t,x),t)>0$ for all $t\in[0, T)$, $x\in\mathbb{R}$.
Since $q_x(t,x)>0$, we have that  $q(t,x)$ is strictly increasing function. Moreover, integrating \eqref{q_x-est} w.r.t. $x$, we also have $\lim_{x\to\pm\infty}q(t,x)=\pm\infty$. Hence $q(x,t)$ is  one-to-one from $\mathbb{R}$ onto $\mathbb{R}$  and thus $m(t,x)>0$ for all $t\in[0, T)$, $x\in\mathbb{R}$.
\section{The case $A_2<A_1$}\label{app:B}

Notice that in this case $\Sigma_2\subset\Sigma_1$ and $\Sigma_0=[-\frac{1}{A_2},-\frac{1}{A_1}]\cup[\frac{1}{A_1},\frac{1}{A_2}]$.

We define $\Phi_i$ and $\tilde \Phi_i$ as in \eqref{eq_phi} and \eqref{eq}, and introduce the scattering matrices $s(\lambda_\pm)$, this time  for $\lambda\in\dot{\Sigma}_2$,
as  matrices relating $\Phi_1$ and $\Phi_2$ (for brevity we keep for it  the same notation $s$):
\begin{subequations}\label{scat<}
\begin{alignat}{4}
&\Phi_1(x,t,\lambda_\pm)=\Phi_2(x,t,\lambda_\pm)s(\lambda_\pm),\qquad\lambda\in\dot{\Sigma}_2\end{alignat}
\end{subequations}
with $\det s(\lambda_\pm)=1$. In turn, $\tilde\Phi_1$ and $\tilde\Phi_2$ are related by
\begin{subequations}\label{scat_<}
\begin{alignat}{4}
&D_1^{-1}(\lambda_\pm)\tilde\Phi_1(x,t,\lambda_\pm)=D_2^{-1}(\lambda_\pm)\tilde\Phi_2(x,t,\lambda_\pm)\eul^{-Q_2(x,t,\lambda_\pm)}s(\lambda_+)\eul^{Q_1(x,t,\lambda_\pm)},\qquad\lambda\in\dot{\Sigma}_2.
\end{alignat}
\end{subequations}

The scattering coefficients $s_{ij}$ can be expressed as in \eqref{scatcoeff}.
However, in this case, \eqref{scatcoeff-a} implies that  $s_{11}(\lambda)$ can be analytically extended to $\mathbb{C}\setminus\Sigma_1$ and defined on the upper and lower parts of $\dot\Sigma_1$, and, since $\Phi_2^{(2)}$ is analytic in $\mathbb{C}\setminus\Sigma_2$ and $\Phi_1^{(2)}$ is defined on the upper and lower sides of $\Sigma_1$, $s_{12}(\lambda)$ can be extended by \eqref{scatcoeff-c} to the lower and upper sides of $\dot{\Sigma}_1$. Thus 
the following relations hold also on $\dot\Sigma_0$:
\begin{subequations}\label{scat-col<}
\begin{alignat}{4}
    \Phi_2^{(2)}(x,t,\lambda_\pm)=s_{11}(\lambda_\pm)\Phi_1^{(2)}(x,t,\lambda_\pm)-s_{12}(\lambda_\pm)\Phi_1^{(1)}(x,t,\lambda_\pm), \quad\lambda\in\dot\Sigma_0.
\end{alignat}
\end{subequations}
and, respectively,
\begin{small}
\begin{subequations}\label{scat-col_<}
\begin{alignat}{4}
   \hspace{-2mm}(D_2^{-1}\Phi_2^{(2)})(x,t,\lambda_\pm)=\tilde s_{11}(x,t,\lambda_\pm)(D_1^{-1}\Phi_1^{(2)})(x,t,\lambda_\pm)-\tilde s_{12}(x,t,\lambda_\pm)(D_1^{-1}\Phi_1^{(1)})(x,t,\lambda_\pm), ~\lambda\in\dot\Sigma_0,
\end{alignat}
\end{subequations}
\end{small} where $\tilde s(x,t,\lambda_\pm):=\eul^{-Q_2(x,t,\lambda_\pm)}s(\lambda_\pm)\eul^{Q_1(x,t,\lambda_\pm)}$.

\subsection{Symmetries}

\ifshort
The symmetries are similar to  the case $A_1<A_2$. In particular, 

\begin{enumerate}
    \item
    \begin{equation}\label{sym_det}
        |s_{11}(\lambda_+)|^2-|s_{12}(\lambda_+)|^2=1,\quad \lambda\in \dot\Sigma_2.
    \end{equation}

    \item 
    \begin{equation}\label{sym_det_cor}
        \big| \frac{s_{12}(\lambda_+)}{s_{11}(\lambda_+)} \big|\leq 1,\quad \lambda\in \dot\Sigma_2
    \end{equation}
\ifshort
\else
    Note that $\big| \frac{s_{12}(\lambda_+)}{s_{11}(\lambda_+)} \big|= 1$ for $\lambda\in \dot\Sigma_2$ iff $s_{11}(\lambda_+)=\infty$.
\fi   
    \item 
    \begin{subequations}\label{sym_s11<}
    \begin{alignat}{3} \label{sym_s11-a<}
    s_{11}(\lambda_+)&=s_{22}(\lambda_-),\quad  \lambda\in\dot\Sigma_2,\\\label{sym_s11-b<}
   s_{11}(\lambda_+)&=\ii s_{12}(\lambda_-),\quad  \lambda\in\dot\Sigma_0,\\\label{sym_s11-c<}
    s_{11}(\lambda_-)&=-\ii s_{12}(\lambda_+),\quad  \lambda\in\dot\Sigma_0.
    \end{alignat}
    \end{subequations}
    
    \item \begin{equation}\label{sym_sigma_0_<}
        \big| \frac{s_{12}(\lambda_+)}{s_{11}(\lambda_+)} \big|=1,\quad \lambda\in \dot\Sigma_0
    \end{equation}
    
    \item \begin{equation}\label{sym_minus<} 
    (D_j^{-1}\tilde\Phi_j)((-\lambda)_-)=-\sigma_3(D_j^{-1}\tilde\Phi_j)(\lambda_+)\sigma_3, \quad \lambda_+\in\dot\Sigma_j.
    \end{equation}
    
    \item 
    \begin{equation}\label{sym_barbar_notin_Sigma_i<}
    \overline{(D_j^{-1}\tilde\Phi_j^{(j)})(\overline{\lambda})}=-(D_j^{-1}\tilde\Phi_j^{(j)})(\lambda), \quad \lambda\in\mathbb{C}\setminus\Sigma_j,
    \end{equation}
    
    \item 
    \begin{subequations}\label{sym_minus_<} 
    \begin{alignat}{4}\label{sym_minus_a<} 
    (D_1^{-1}\tilde\Phi_1^{(1)})(-\lambda)&=-\sigma_3(D_1^{-1}\tilde\Phi_1^{(1)})(\lambda), \quad \lambda\in\mathbb{C}\setminus\Sigma_1,\\ \label{sym_minus_b<} 
    (D_2^{-1}\tilde\Phi_2^{(2)})(-\lambda)&=\sigma_3(D_2^{-1}\tilde\Phi_2^{(2)})(\lambda), \quad \lambda\in\mathbb{C}\setminus\Sigma_2.
    \end{alignat}
    \end{subequations}
    
    \item 
    \begin{subequations}\label{sym_D-tilde-phi<} 
    \begin{alignat}{4}\label{sym_D-tilde-phi-a<}
    D_j^{-1}(\lambda_-)\tilde\Phi_j^{(j)}(\lambda_-)&=(-\ii D_j^{-1}(\lambda_+)\tilde\Phi_j(\lambda_+)\sigma_1)^{(j)},\quad \lambda\in\dot\Sigma_2,\\ \label{sym_D-tilde-phi-b<}
   D_1^{-1}(\lambda_-)\tilde\Phi_1^{(1)}(\lambda_-)&=(-\ii D_1^{-1}(\lambda_+)\tilde\Phi_1(\lambda_+)\sigma_1)^{(1)},\quad \lambda\in\dot\Sigma_0,\\ \label{sym_D-tilde-phi-c<}
   D_2^{-1}(\lambda_-)\tilde\Phi_2^{(2)}(\lambda_-)&=D_2^{-1}(\lambda_+)\tilde\Phi_2^{(2)}(\lambda_+),\quad \lambda\in\dot\Sigma_0.
\end{alignat}
\end{subequations}

\end{enumerate}

\fi


\subsection{Discrete spectrum} It can be shown in a similar way as for the case $A_1<A_2$ that discrete spectrum is located on $(-\frac{1}{A_1},\frac{1}{A_1})$ (assuming that spectral singularities do not arise in the branch points).

\subsection{RH problem parametrized by $\BS{(x,t)}$}

\begin{notations*}
In this case it is convenient to introduce $\check \rho$ as

\begin{equation}\label{rho<}
    \check\rho(\lambda)=\frac{s_{12}(\lambda_+)}{s_{11}(\lambda_+)},\quad\lambda\in\dot\Sigma_2\cup\dot\Sigma_0.
\end{equation}
Observe that \eqref{sym_det_cor} and \eqref{sym_sigma_0_<} imply that 
\begin{subequations}\label{prop-rho<}
\begin{alignat}{4}\label{prop-rho-a<}
|\check\rho(\lambda)|\leq 1,\quad\lambda\in\dot\Sigma_2,\\\label{prop-rho-b<}
|\check\rho(\lambda)|= 1,\quad\lambda\in\dot\Sigma_0.
\end{alignat}
\end{subequations}
\end{notations*}
Recalling the analytic properties of eigenfunctions and scattering coefficients, we introduce the matrix-valued function 

\begin{equation}\label{M<}
N(x,t,\lambda)= \left( (D_1^{-1}\tilde\Phi_1^{(1)})(x,t,\lambda),\frac{(D_2^{-1}\tilde\Phi_2^{(2)})(x,t,\lambda)}{s_{11}(\lambda)\eul^{p_1(x,t,\lambda)-p_2(x,t,\lambda)}}\right),\quad\lambda\in\mathbb{C}\setminus\Sigma_2,
\end{equation}
meromorphic in $\mathbb{C}\setminus\Sigma_2$,
where $p_j$, $j=1,2$, are defined in \eqref{p_i}.
Since $D_j^{-1}(\lambda)\tilde\Phi_j(x,t,\lambda)=\Phi_j(x,t,\lambda)\eul^{Q_j(x,t,\lambda)}$, $N(x,t,\lambda)$ can be written as 
\[
N(x,t,\lambda)= \left( \Phi_1^{(1)}(x,t,\lambda),\frac{\Phi_2^{(2)}(x,t,\lambda)}{s_{11}(\lambda)}\right)\eul^{p_1(x,t,\lambda)\sigma_3}.
\]

Proceeding as in case $A_1<A_2$, we conclude that $N(x,t,\lambda)$ can be characterized as the solution 
of the  following Riemann-Hilbert problem: 

Find a $2\times 2$ meromorphic matrix $N(x,t,\lambda)$ that satisfies the following conditions:
\begin{enumerate}[\textbullet]
\item The \emph{jump} condition
\begin{subequations}\label{RH-x-J0<}
\begin{equation}\label{RH-x<}
N^+(x,t,\lambda)=N^-(x,t,\lambda)G(x,t,\lambda),\quad \lambda\in\dot\Sigma_2\cup\dot\Sigma_0,
\end{equation}
where
\begin{equation}\label{J<}
G(x,t,\lambda)=
\begin{pmatrix}
0&\ii\\
\ii&0
\end{pmatrix}\begin{pmatrix}
\eul^{-p_1(\lambda_+)}&0\\
0&\eul^{p_1(\lambda_+)}
\end{pmatrix}G_0(\lambda)
\begin{pmatrix}
\eul^{p_1(\lambda_+)}&0\\
0&\eul^{-p_1(\lambda_+)}
\end{pmatrix}
\end{equation}
with
\begin{equation}\label{J0<}
G_0(\lambda)=\begin{cases} \begin{pmatrix}
1&-\check\rho(\lambda)\\
\overline{\check\rho(\lambda)}&1-|\check\rho(\lambda)|^2
\end{pmatrix}
, \quad \lambda\in\dot\Sigma_2,\\
\begin{pmatrix}
1&-\check\rho(\lambda)\\
\frac{1}{\check\rho(\lambda)}&0
\end{pmatrix}, \quad \lambda\in\dot\Sigma_0.
\end{cases}
\end{equation}
\end{subequations}

\item
The \emph{normalization} condition:
\begin{equation}\label{norm<}
N(x,t,\lambda)=\begin{cases}\sqrt{\frac{1}{2}}
\begin{pmatrix}
-1&\ii\\
\ii&-1
\end{pmatrix}+\ord(\frac{1}{\lambda}),\quad \lambda\to\infty, \quad \lambda\in\mathbb{C}^+,\\
\sqrt{\frac{1}{2}}
\begin{pmatrix}
1&\ii\\
\ii&1
\end{pmatrix} +\ord(\frac{1}{\lambda}),\quad \lambda\to\infty, \quad \lambda\in\mathbb{C}^-,
\end{cases}
\end{equation}

\item
\emph{Singularity} conditions: the singularities of 
$N(x,t,\lambda)$ at  $\pm\frac{1}{A_j}$ are of order not bigger than $\frac{1}{4}$.
\item
\emph{Residue} conditions (if any): given $\accol{\check\lambda_k,  \check\kappa_k}_1^{\check N}$ with $\check\lambda_k\in (0,\frac{1}{A_1})$
and $\check\kappa_k\in\mathbb{R}\setminus\{0\}$, 
$N^{(2)}(x,t,\lambda)$ has simple poles at $\accol{\check\lambda_k, -\check\lambda_k}_1^{\check N}$, with the residues  satisfying the equations

\begin{equation}
\label{res-M-+<}
    \Res_{\pm\check\lambda_k}N^{(2)}(x,t,\lambda)=\check\kappa_k\eul^{-2p_1(\check\lambda_k)}N^{(2)}(x,t,\pm\check\lambda_k).
\end{equation}
\end{enumerate}

\begin{remark}
The solution of the RH problem above, if exists, satisfies the following properties:
\begin{enumerate}
\item $\det N\equiv 1$.
 
\item 
\emph{Symmetries}:
\begin{subequations}\label{sym-M<}
\begin{alignat}{4}
N(-\lambda)&=-\sigma_3 N(\lambda)\sigma_3,\qquad \overline{N(\overline{\lambda})}=-N(\lambda),\qquad \lambda\in\mathbb{C}\setminus\Sigma_1,\\
N((-\lambda)_-)&=-\sigma_3 N(\lambda_+)\sigma_3,\qquad \overline{N(\lambda_-)}=-N(\lambda_+),\qquad \lambda\in\dot\Sigma_2.
\end{alignat}
\end{subequations}
where $N(\lambda)\equiv N(x,t,\lambda)$
(follows from the respective symmetries of the jump matrix and the residue conditions,
assuming the uniqueness of the solution).
\end{enumerate}
\end{remark}

\subsection{Eigenfunctions near $\lambda=0$}
Introducing $\tilde \Phi_{0,j}$ as in \eqref{inteq0_lam} and proceeding as in case $A_1<A_2$,  the following development of $N(x,t,\lambda)$ near $\lambda=0$ holds:
\begin{equation}\label{x-expan_N} 
N(x,t,\lambda)=\ii\begin{pmatrix}
0& b_1(x,t)\\ b_1^{-1}(x,t)&0
\end{pmatrix}+\ii\lambda\begin{pmatrix}
b_2(x,t)&0\\0&b_3(x,t)
\end{pmatrix}+\ord(\lambda^2),
\end{equation}
where 
\begin{subequations}\label{b_j}
    \begin{alignat}{4}
    b_1(x,t) &= \eul^{\frac{1}{2A_1
}\int^x_{-\infty}( m(\xi,t)-A_1)\dd\xi},\\
b_2(x,t) &= (\int_{-\infty}^x e^{-(x-\xi)}\frac{m-A_1}{2}\dd\xi+\frac{A_1}{2})\eul^{-\frac{1}{2A_1
}\int^x_{-\infty}( m(\xi,t)-A_1)\dd\xi},\\
b_3(x,t) &= (\int^{\infty}_x e^{(x-\xi)}\frac{m-A_2}{2}\dd\xi+\frac{A_2}{2})\eul^{\frac{1}{2A_1
}\int^x_{-\infty}( m(\xi,t)-A_1)\dd\xi}.
    \end{alignat}
\end{subequations}

\begin{proposition}\label{prop-recover-x-2}
 $u(x,t)$ and $u_x(x,t)$ can be algebraically expressed in terms of the coefficients
 $b_j(x,t)$, $j=1,3$ in the development \eqref{x-expan_N} of $N(x,t,\lambda)$ as follows:
 \begin{subequations}\label{uu_x_}
    \begin{alignat}{4}
    u(x,t) &= b_1(x,t)b_2(x,t)+ b_1^{-1}(x,t) b_3(x,t),\\
    \label{uu_x-b_}
    u_x(x,t) & = -b_1(x,t)b_2(x,t)+ b_1^{-1}(x,t)b_3(x,t).
    \end{alignat}
\end{subequations}
 \end{proposition}

\subsection{RH problem in the $\BS{(y,t)}$ scale}
Introducing the new space variable $\check y(x,t)$  by
\begin{equation}\label{shkala<}
    \check y(x,t)=x+\frac{1}{A_1}\int^x_{-\infty}(m(\xi,t)-A_1)\dd\xi-A_1^2 t
\end{equation} 
and introducing $\hat N(\check y,t,\lambda)$ so that $N(x,t,\lambda)=\hat N (\check y(x,t),t,\lambda)$, the jump condition \eqref{RH-x<} becomes
\begin{subequations} \label{Jp-y<}
\begin{equation}\label{jump-y<}
\hat N^+(\check y,t,\lambda)=\hat N^-(\check y,t,\lambda)\hat G(\check y,t,\lambda),\quad\lambda\in \dot\Sigma_2\cup\dot\Sigma_0,
\end{equation}
where
\begin{equation}\label{J-J0<}
\hat G(\check y,t,\lambda)\coloneqq\begin{pmatrix}
0&\ii\\
\ii&0
\end{pmatrix}\begin{pmatrix}
\eul^{-\hat p_1(\check y,t,\lambda_+)}&0\\
0&\eul^{\hat p_1(\check y,t,\lambda_+)}
\end{pmatrix}G_0(\lambda)
\begin{pmatrix}
\eul^{\hat p_1(\check y,t,\lambda_+)}&0\\
0&\eul^{-\hat p_1(\check y,t,\lambda_+)}
\end{pmatrix},
\end{equation}
 $G_0(\lambda)$ is defined by \eqref{J0<},
\begin{equation}\label{p-y<}
\hat p_1(\check y,t,\lambda) \coloneqq \frac{\ii A_1 k_1(\lambda)}{2}\left(\check y-\frac{2t}{\lambda^2}\right).
\end{equation}
\end{subequations}
Thus $G(x,t,\lambda)=\hat G(\check y(x,t),t,\lambda)$ and $p_1(x,t,\lambda)=\hat p_1(\check y(x,t),t,\lambda)$, where the jump $G(x,t,\lambda)$ and the phase $p_1(x,t,\lambda)$ are defined in \eqref{J<} and \eqref{p_i}, respectively.

Accordingly,  the residue conditions \eqref{res-M-+<} become
\begin{equation}\label{res-M-hat<}
\begin{split}
\Res_{\pm\check \lambda_k}\hat N^{(2)}(\check y,t,\lambda)&=\check \kappa_k\eul^{-2\hat p_1(\check y,t,\lambda_k)}\hat N^{(1)}(\check y,t,\pm\check \lambda_k),
\end{split}
\end{equation}
with $\check \kappa_k=\frac{1}{\check b_k s'_{11}(\check \lambda_k)}$.

Noticing that the normalization condition \eqref{norm<}, the symmetries \eqref{sym-M<}, and the singularity conditions at $\lambda=\pm \frac{1}{A_j}$ hold in the new scale $(\check y,t)$, we arrive at the basic RH problem.

\begin{rh-pb*}
Given $\check\rho(\lambda)$ for $\lambda\in \dot\Sigma_2\cup\dot\Sigma_0$, and $\accol{\check \lambda_k,\check \kappa_k}_1^{\check N}$ with $\check \lambda_k\in(0,\frac{1}{A_1})$ and $\check\kappa_k\in\mathbb{R}\setminus\{0\}$, 
associated with the initial data $u_0(x)$ in \eqref{mCH1-ic},
find a piece-wise (w.r.t.~$\dot\Sigma_1$) meromorphic, $2\times 2$-matrix valued function $\hat N(\check y,t,\lambda)$ satisfying the following conditions:

\begin{enumerate}[\textbullet]
\item
The jump condition \eqref{Jp-y<} across $\dot\Sigma_2\cup\dot\Sigma_0$ (with $G_0(\lambda)$ defined by \eqref{J0<}).
\item
The residue conditions \eqref{res-M-hat<}.
\item The \emph{normalization} condition:
\begin{equation}\label{norm-hat<}
\hat N(\check y,t,\lambda)=\begin{cases}\sqrt{\frac{1}{2}}
\begin{pmatrix}
-1&\ii\\
\ii&-1
\end{pmatrix}+\ord(\frac{1}{\lambda}), \quad \lambda\to\infty,~ \lambda\in\mathbb{C}^+,\\
\sqrt{\frac{1}{2}}
\begin{pmatrix}
1&\ii\\
\ii&1
\end{pmatrix}+\ord(\frac{1}{\lambda}), \quad \lambda\to\infty,~ \lambda\in\mathbb{C}^-.
\end{cases}
\end{equation}

\item
\emph{Singularity} conditions:
$\hat N(\check y,t,\lambda)$ may have singularities at $\pm\frac{1}{A_j}$ of order $\frac{1}{4}$.

\item \emph{Symmetries}:
\begin{subequations}\label{sym-M-hat<}
\begin{alignat}{4}
\hat N(-\lambda)&=-\sigma_3 \hat N(\lambda)\sigma_3,\qquad \overline{\hat N(\overline{\lambda})}=-N(\lambda),\qquad \lambda\in\mathbb{C}\setminus\Sigma_2,\\
\hat N((-\lambda)_-)&=-\sigma_3 \hat N(\lambda_+)\sigma_3,\qquad \overline{\hat N(\lambda_-)}=-\hat N(\lambda_+),\qquad \lambda\in\dot\Sigma_1.
\end{alignat}
\end{subequations}
where $\hat N(\lambda)\equiv \hat N(\check y,t,\lambda)$.
\end{enumerate}
\end{rh-pb*}

\subsection{Recovering $\BS{u(x,t)}$ from the solution of the RH problem}\label{sec:recover<}


\begin{theorem}\label{main_N}
Assume that $u(x,t)$ is the solution of the Cauchy problem \eqref{mCH1-ic} and let  $\hat N(\check y,t,x)$ be the solution of the associated RH problem \eqref{Jp-y<}--\eqref{norm-hat<}, whose data are determined by $u_0(x)$.
Let  
\begin{equation}\label{M-hat-expand_N}
\hat N(\check y,t,\lambda)=\ii\begin{pmatrix}
0& \hat b_1(\check y,t)\\\hat b_1^{-1}(\check y,t)&0
\end{pmatrix}+\ii\lambda\begin{pmatrix}
\hat b_2(\check y,t)&0\\0&\hat b_3(\check y,t)
\end{pmatrix}+\ord(\lambda^2)
\end{equation}
be the development of  $\hat N(\check y,t,x)$ at $\lambda=0$. 
Then the solution $u(x,t)$ of the Cauchy problem \eqref{mCH1-ic} 
can be expressed, in a parametric form, in terms of $\hat b_j(\check y,t)$, $j=1,2,3$:
$u(x,t)=\hat u(\check y(x,t),t)$, where
\begin{subequations}\label{recover-2_N}
\begin{align}\label{u_(y,t)_N}
&\hat u(\check y,t)=\hat b_1(\check y,t)\hat b_2(\check y,t)+\hat b_1^{-1}(\check y,t)\hat b_3(\check y,t),\\
&x(\check y,t)=\check y-2\ln\hat b_1(\check y,t)+A_2^2 t.
\label{x(y,t)-2_N}
\end{align}
Additionally, $\hat u_x(\check y,t)$ can also be algebraically expressed in terms of $\hat b_j(\check y,t)$, $j=1,2,3$: 
$u_x(x,t)=\hat u_x(\check y(x,t),t)$, where
\begin{equation}
\label{u_x(y,t)-2_N}
\hat u_x(\check y,t)=-\hat b_1(\check y,t)\hat b_2(\check y,t)+\hat b_1^{-1}(\check y,t)\hat b_3(\check y,t).
\end{equation}
\end{subequations}
\end{theorem}




\end{document}